\documentclass{article}

\usepackage[margin=3cm,bottom=3cm]{geometry}
\usepackage{amsmath}
\usepackage{amssymb}
\usepackage{amsthm}
\usepackage{enumerate}
\usepackage{hyperref}
\usepackage[all]{xy}
\usepackage{theoremref}
\usepackage{tikz-cd}
\usepackage{color}

\theoremstyle{definition}
\newtheorem{thm}{Theorem}[section]
\newtheorem{defn}[thm]{Definition}
\newtheorem{example}[thm]{Example}
\newtheorem{q}[thm]{Question}
\newtheorem{prop}[thm]{Proposition}
\newtheorem*{setup}{Set-up}
\newtheorem{remark}[thm]{Remark}
\newtheorem{notation}[thm]{Notation}
\newtheorem{algorithm}[thm]{Algorithm}

\numberwithin{subcase}{case}

\newtheorem*{mainproblem}{Main problem}
\newtheorem*{Zproblem}{Zariski's problem}
\newtheorem*{Cproblem}{Classification problem}

\author{Eloise Hamilton}

\title{Classifying complete $\mathbb{C}$-subalgebras of $\mathbb{C}[[t]]$}
\date{ }

\begin{document}
\maketitle

\begin{abstract}

We address the problem of classifying complete $\mathbb{C}$-subalgebras of $\mathbb{C}[[t]]$. A discrete invariant for this classification problem is the semigroup of orders of the elements in a given $\mathbb{C}$-subalgebra. Hence we can define the space $\mathcal{R}_{\Gamma}$ of all $\mathbb{C}$-subalgebras of $\mathbb{C}[[t]]$ with semigroup $\Gamma$. After relating this space to the Zariski moduli space of curve singularities and to a moduli space of global singular curves, we prove that $\mathcal{R}_{\Gamma}$ is an affine variety by describing its defining equations in an ambient affine space in terms of an explicit algorithm.   
Moreover, we identify certain types of semigroups $\Gamma$ for which $\mathcal{R}_{\Gamma}$ is always an affine space, and for general $\Gamma$ we describe the stratification of $\mathcal{R}_{\Gamma}$ by embedding dimension. We also describe the natural map from $\mathcal{R}_{\Gamma}$ to the Zariski moduli space in some special cases. Explicit examples are provided throughout. 
\end{abstract}

\tableofcontents

\section*{Introduction}

In this paper we consider the following algebraic problem: the classification of complete $\mathbb{C}$-subalgebras of the ring of formal power series in one variable $\mathbb{C}[[t]]$. As is often the case for classification problems in algebraic geometry, the problem can be broken down into two steps. First, we search for a discrete invariant which provides an initial coarse classification of the objects. Then, for each fixed value of the invariant, we search for an algebraic variety which parametrises all objects with this given value.  A discrete invariant for our problem is given by a semigroup in $\mathbb{N}$, obtained by taking the orders of elements of a given $\mathbb{C}$-subalgebra of $\mathbb{C}[[t]]$.

\begin{defn} \thlabel{defnsemigp}
Let $R$ be a $\mathbb{C}$-subalgebra of $\mathbb{C}[[t]]$. The \emph{semigroup} of $R$ is the set $\Gamma_R \subseteq \mathbb{N}$ defined by $$\Gamma_R := \{ n \in \mathbb{N}  \ | \ \exists  \ f \in R^{\ast} \text{ with } \operatorname{ord} f = n\}.$$ 
\end{defn} 

This set has the structure of a semigroup since if $f,g \in R$ have orders $n$ and $m$ respectively, then $fg \in R$ has order $n + m$. For example, the semigroup of the $\mathbb{C}$-subalgebra $R = \mathbb{C}[[t^2,t^5]]$ is the semigroup $\Gamma_R$ generated by $2$ and $5$, which we denote $\langle 2,5\rangle$. 

The semigroup of a $\mathbb{C}$-subalgebra $R$ of $\mathbb{C}[[t]]$ is indeed an invariant of our classification problem as it is computed directly from the elements of $R$. Moreover, any semigroup $\Gamma \subseteq \mathbb{N}$ gives rise to a complete $\mathbb{C}$-subalgebra of $\mathbb{C}[[t]]$ simply by taking the $\mathbb{C}$-subalgebra generated by all monomials of the form $t^n$ for $n \in \Gamma$. Thus the problem of classifying complete $\mathbb{C}$-subalgebras of $\mathbb{C}[[t]]$ can be reduced to the problem of classifying complete $\mathbb{C}$-subalgebras of $\mathbb{C}[[t]]$ with a given semigroup $\Gamma \subseteq \mathbb{N}$. This is the guiding problem of this paper. 

\begin{mainproblem} \thlabel{theproblem}
For a given semigroup $\Gamma \subseteq \mathbb{N}$, describe the space $$\mathcal{R}_{\Gamma} : = \{\text{complete $\mathbb{C}$-subalgebras of $\mathbb{C}[[t]]$ with semigroup $\Gamma$}\}.$$
\end{mainproblem}

We will describe $\mathcal{R}_{\Gamma}$ for a particular type of semigroup, so-called numerical semigroups. In the above example, the ring $R$ has semigroup $\Gamma_R = \langle 2, 5\rangle = \{2,4,5,6,7,8, \hdots\}$. In this case, $n \in \Gamma_R$ for all $n \geq 4$, but $3 \notin \Gamma_R$. In general, a semigroup $\Gamma$ in $\mathbb{N}$ containing an element $c$ with $c-1 \notin \mathbb{N}$ but $n \in \Gamma$ for all $n \geq c$ is called a \emph{numerical} semigroup, and $c$ is called the \emph{conductor} of $\Gamma$. Our reason for restricting our study to numerical semigroups and for working over the complex numbers is motivated by geometry. A unibranch singularity on a curve defined over the complex numbers gives rise to a $\mathbb{C}$-subalgebra of $\mathbb{C}[[t]]$ via its complete local ring (see Section \ref{zariski}), and thus to a semigroup in $\mathbb{N}$. Semigroups arising from unibranch curve singularities are exactly the numerical semigroups (see \cite[Section 2]{Barucci2003}), and that is why we are primarily interested in describing $\mathcal{R}_{\Gamma}$ for such semigroups. From here on, we will always assume that our semigroups are numerical. While we will work over $\mathbb{C}$ throughout, all results from Section \ref{affinevariety} hold true for an algebraically closed field of arbitrary characteristic.  

The structure of the paper is as follows. The aim of Section \ref{motivation} is to motivate the study of the space $\mathcal{R}_{\Gamma}$ by showing how it relates to two important and previously studied classification problems: one relating to curve singularities (Section \ref{zariski}), the other to global singular curves (Section \ref{global}). In Section \ref{affinevariety} we will show that in the case of numerical semigroups, the set $\mathcal{R}_{\Gamma}$ is in bijection with the points of an affine variety (\thref{step3}). We will prove this directly by providing an algorithm which, given a semigroup $\Gamma$, determines the polynomials $g_1,\hdots, g_n \in \mathbb{C}[x_1,\hdots, x_M]$ such that $\mathcal{R}_{\Gamma} = V(g_1,\hdots, g_n) \subseteq \mathbb{C}^M$. In Section \ref{manyexamples} we will work through some examples to show how, in practice, these results can be used to explicitly compute $\mathcal{R}_{\Gamma}$. In Section \ref{properties} we will investigate properties of the space $\mathcal{R}_{\Gamma}$. Section \ref{alwaysaffine} addresses the question of whether or not $\mathcal{R}_{\Gamma}$ can always be identified with an affine space.  Section \ref{stratificatn} describes how the space $\mathcal{R}_{\Gamma}$ admits a stratification by locally closed subsets, corresponding to subalgebras with a fixed number of generators. This stratification is finite and starts with $\mathcal{R}_{\Gamma}^{\text{plane}}$, the subset of $\mathcal{R}_{\Gamma}$ consisting of subalgebras which can be generated by just two elements (these correspond geometrically to plane curve singularities). Finally, Section \ref{mapmtor} explores the relationship between $\mathcal{R}_{\Gamma}$ and the Zariski moduli space $\mathcal{M}_{\Gamma}$, viewed as the quotient of $\mathcal{R}_{\Gamma}$ by the action of the automorphism group of $\mathbb{C}[[t]]$. We will explicitly compute the quotient map in two cases. 

After completion of this paper, a paper \cite{Ishii1980} by Ishii from 1980 was brought to my attention. This earlier paper considers the same algebraic problem addressed in the present paper, and the conclusions overlap: Theorem 3 of \cite{Ishii1980} shows that the moduli functor associated with the classification problem is representable by an affine scheme (cf.\ \thref{step3}), and Corollaries 4 and 5 identify this affine scheme as an affine space in the same two cases which we consider in \thref{2gen,onegen}. However, the perspective taken in the present paper is different. Firstly, we emphasise the links between the space $\mathcal{R}_{\Gamma}$ and moduli spaces involving curve singularities on the one hand and global singular curves on the other, while the perspective in \cite{Ishii1980} is purely algebraic. Secondly, we adopt a concrete algorithmic approach to the problem, similar to that in \cite{Zariski1965} and in \cite{HH2007}, which yields an explicit description of the generators for the affine variety $\mathcal{R}_{\Gamma}$ solely in terms of the data of the semigroup $\Gamma$. This contrasts with the more abstract and scheme-theoretic nature of the results in \cite{Ishii1980}. It should also be noted that \cite{Ishii1980} does not answer \thref{q1} regarding whether or not the space $\mathcal{R}_{\Gamma}$ is always an affine space for semigroups with three generators, and more generally which semigroups give rise to an affine space. Thus this question remains open.

\paragraph{Acknowledgements.} I would like to thank David Smyth for suggesting this problem to me, for his continued support and guidance, and for allowing me to complete part of this work as a Research Assistant to him funded under an Australian Research Council grant (number DE140100259). I would also like to thank Joshua Jackson for bringing \cite{Ishii1980} to my attention. Finally, I am very grateful to an anonymous referee for their insight and helpful suggestions.

\section{Geometric relevance of the space $\mathcal{R}_{\Gamma}$} \label{motivation}

\subsection{$\mathcal{R}_{\Gamma}$ and the Zariski moduli space $\mathcal{M}_{\Gamma}$ of curve singularities} \label{zariski}

Given a  curve $X$ (a one-dimensional abstract variety) and a unibranch singularity $p \in X$, we can associate to $p$ its complete local ring $\widehat{\mathcal{O}}_{X,p}$. Let $\widetilde{p}$ denote the preimage of $p$ under the normalisation map $\pi: \widetilde{X} \rightarrow X$. The map $\pi$ induces an injection $$\widehat{\mathcal{O}}_{X,p} \hookrightarrow \widehat{\mathcal{O}}_{\widetilde{X},\widetilde{p}} \cong  \mathbb{C}[[t]]$$ where we fix an isomorphism $\widehat{\mathcal{O}}_{\widetilde{X}, \widetilde{p}} \cong \mathbb{C}[[t]]$ using the Cohen Structure Theorem. Thus we can identify $\widehat{\mathcal{O}}_{X,p}$, which we denote $\mathcal{O}$ for simplicity, as a $\mathbb{C}$-subalgebra of $\mathbb{C}[[t]]$. The \emph{semigroup} of the singularity is the semigroup $\Gamma_{\mathcal{O}}$ as introduced in \thref{defnsemigp}. This is well-defined because automorphisms of $\mathbb{C}[[t]]$ correspond to power series of order $1$, and these preserve orders.

Both the complete local ring of a curve singularity and its semigroup are important geometric invariants, the former a continuous invariant and the second a discrete invariant. Indeed, the complete local ring encodes the analytic type of a curve singularity, while in the case of plane curve singularities the semigroup encodes its topological type. 

\begin{defn}
Let $\mathcal{C}$ and $\mathcal{C}'$ be two curves embedded in $\mathbb{C}^n$ with singularities at the origin. These singularities are \emph{topologically equivalent} (respectively \emph{analytically equivalent}) if there exist neighbourhoods $U$ and $U'$ of $0$ in $\mathbb{C}^n$ and a homeomorphism (respectively analytic isomorphism) $\phi: U \rightarrow U'$ such that $\phi(U \cap \mathcal{C}) = U' \cap \mathcal{C}'$. 
\end{defn}

It is well known that two curve singularities are analytically equivalent if and only if they have isomorphic complete local rings \cite[Theorem 1.3]{Grothendieck1960-1961}. For plane curve singularities, the fact that the topological type is encoded by its semigroup is much more surprising. It was proven by Zariski in 1965 for unibranch plane curve singularities, and subsequently generalised to all\footnote{Note that for multibranch curve singularities, the semigroup is a subset of $\mathbb{N}^r$, where $r$ is the number of branches of the singularity.} plane curve singularities by Waldi in 1973 \cite{Zariski1965,Waldi1972}. The topological significance of the semigroup in the non-planar case does not appear to have been studied. Nevertheless, both in the planar and non-planar case, the semigroup is a discrete invariant and thus we can consider the problem of classifying curve singularities with a given semigroup up to analytic equivalence. 

This problem was first considered in the case of unibranch plane curve singularities by Zariski in \cite{Zariski1986}, but can be stated for unibranch curve singularities of arbitrary embedding dimension.

 \begin{Zproblem}
 
Let $\Gamma$ be a numerical semigroup. Describe the space $$\mathcal{M}_{\Gamma} = \{\text{unibranch curve singularities with semigroup $\Gamma$}\}  \ / \text{ analytic equivalence}.$$ 

\end{Zproblem}
The space $\mathcal{M}_{\Gamma}$ is called the \emph{Zariski moduli space}. Zariski considers this problem in \cite{Zariski1965} in the case of plane curve singularities, which he studies via their parametrisations. The complete local ring $\mathcal{O}$ of any unibranch plane curve singularity can be identified as a quotient $\mathbb{C}[[x,y]]/(f)$ for some irreducible power series $f \in \mathbb{C}[[x,y]]$ converging in a neighbourhood of the origin in $\mathbb{C}^2$. As seen above, there is an injection from $\mathcal{O} \cong \mathbb{C}[[x,y]]/(f)$ into $\mathbb{C}[[t]]$ via the normalisation map. If we denote by $x(t)$ and $y(t)$ the power series corresponding to the images of the elements $[x]$ and $[y]$ in $\mathbb{C}[[x,y]]/(f)$ under this injection, then $\mathcal{O} \cong \mathbb{C}[[x(t),y(t)]] \subseteq \mathbb{C}[[t]]$. 

By definition the power series $x(t)$ and $y(t)$ satisfy the property that $f(x(t),y(t)) = 0$. The pair $(x(t),y(t))$ is called a \emph{parametrisation} of the singularity. A parametrisation $(x(t),y(t))$ defines a homomorphism $\varphi: \mathbb{C}[[x,y]] \rightarrow \mathbb{C}[[t]]$ given by $x \mapsto x(t)$ and $y \mapsto y(t)$. For simplicity, we often identify a parametrisation $(x(t),y(t))$ with the corresponding homomorphism $\varphi$ and write $\varphi = (x(t),y(t))$. The semigroup of a parametrisation $\varphi = (x(t),y(t))$ is just the semigroup of the ring $\varphi(\mathbb{C}[[x,y]]) = \mathbb{C}[[x(t),y(t)]]$.

The notion of analytic equivalence can easily be transferred to parametrisations. Two parametrisations $\varphi = (x(t),y(t))$ and $\varphi' = (x'(t),y'(t))$ define analytically equivalent singularities if and only if there exist automorphisms $\rho \in \operatorname{Aut} \mathbb{C}[[t]]$ and $\sigma \in \operatorname{Aut}\mathbb{C}[[x,y]]$ such that the following diagram commutes: 
$$ \xymatrix{ 
\mathbb{C}[[x,y]] \ar[r]^{\varphi}  & \mathbb{C}[[t]] \\
\mathbb{C}[[x,y]] \ar[u]^{\sigma}  \ar[r]_{\varphi'} & \mathbb{C}[[t]] \ar[u]_{\rho}. 
}
$$
If such automorphisms exist, then the parametrisations are said to be $\mathcal{A}$-equivalent and we write $\varphi \sim_{\mathcal{A}} \varphi'$. Thus the classification of plane curve singularities with semigroup $\Gamma$ up to analytic equivalence amounts to describing the set of parametrisations with semigroup $\Gamma$, denoted $\Sigma_{\Gamma}$, up to $\mathcal{A}$-equivalence. This equivalence relation is given by the group action of $\operatorname{Aut} \mathbb{C}[[t]] \times \operatorname{Aut} \mathbb{C}[[x,y]]$ on $\Sigma_{\Gamma}$ defined by $(\rho,\sigma) \cdot \varphi  = \sigma \circ \varphi \circ \rho^{-1}$ where $\varphi \in \Sigma_{\Gamma}$. 

Zariski's approach in \cite{Zariski1965} to describing $\Sigma_{\Gamma} / \sim_{\mathcal{A}}$ is to find, for each parametrisation $\varphi \in \Sigma_{\Gamma}$, the simplest parametrisation in its orbit under the above group action. Note that given a curve singularity defined by $f$, finding an explicit parametrisation $\varphi$ to start with is a non-trivial problem. This can be done using the Newton-Puiseaux method which, given $f \in \mathbb{C}[[x,y]]$, produces a pair $(t^N, y(t))$ where $y \in \mathbb{C}[[t]]$ satisfying $f(t^N, y(t)) =0$. Such parametrisations are called \emph{Puiseaux parametrisations}. The first part of \cite{Zariski1965} is devoted to finding various ingenious elimination criteria which provide ways of simplifying a given Puiseaux parametrisation $(t^N,y(t))$ whilst preserving $\mathcal{A}$-equivalence. In 2007, building on Zariski's results, Hefez and Hernandez obtained a complete set of elimination criteria, thus providing an explicit set-theoretic description of $\mathcal{M}_{\Gamma}$ in the case of plane curve singularities. The complete set of elimination criteria appears in \cite{HH2007}. 

After equipping $\mathcal{M}_{\Gamma}$ with a suitable topology, Zariski shows that the space is not in general separated. In order to obtain a separated moduli space, one must restrict to so-called ``general branches'', defined in terms of the dimension of their module of deformations \cite{Zariski1986,Washburn1988}. The corresponding subset $\mathcal{M}_{\Gamma}^{\text{gen}}$ is open and dense in $\mathcal{M}_{\Gamma}$. The main questions posed by Zariski are whether this subset is an algebraic variety, and if so whether an explicit formula for its dimension can be determined. Zariski answers these questions in a number of special cases (all of which correspond to semigroups with just two generators) through explicit and detailed calculations. In 1978, building on Zariski's calculations, Delorme obtained in \cite{Delorme1978} an explicit formula for the dimension of $\mathcal{M}_{\Gamma}^{\text{gen}}$ in the case of semigroups $\Gamma = \langle v_0, v_1 \rangle$. In 1988, Laudal, Martin and Pfister described the structure of $\mathcal{M}_{\Gamma}$ as an algebraic variety for such $\Gamma = \langle v_0, v_1 \rangle$. More precisely, they showed that $\mathcal{M}_{\Gamma}$ in this case admits a stratification defined by fixing the Tjurina number $\tau$ of the singularity, and moreover that on each open stratum a good quotient exists and that it is a quasi-smooth algebraic variety \cite{Laudal1988}. Two years later, Luengo and Pfister obtained an explicit description for $\mathcal{M}_{\Gamma}$ for $\Gamma = \langle 2p, 2q, 2pq+d \rangle$ with $p < q $, $\operatorname{gcd}(p,q)=1$ and $d$ odd: it is the quotient of an affine space $\mathbb{C}^N$ by a suitable action of the group $\mu_d$ of $d$-roots of unity \cite{Luengo1990}. However, these results aside, very little is known about the space $\mathcal{M}_{\Gamma}$, neither its dimension for general semigroups $\Gamma$, nor whether or not it is irreducible. 

It may be possible to further our understanding of the space $\mathcal{M}_{\Gamma}$ by viewing it as the quotient of $\mathcal{R}_{\Gamma}$ by the automorphism group of $\mathbb{C}[[t]]$. We do so explicitly in Section \ref{mapmtor} in two examples.

\subsection{$\mathcal{R}_{\Gamma}$ and the moduli space of global singular curves} \label{global}

The geometric relevance of $\mathcal{R}_{\Gamma}$ is not limited to its relation to the Zariski moduli space $\mathcal{M}_{\Gamma}$. We will see in this section that the space $\mathcal{R}_{\Gamma}$ parametrises the different ways in which a singular point can be ``glued on'' to a given smooth curve. 

We can associate to any given abstract curve $X$ with just one singular point $p \in X$ its normalisation $\widetilde{X}$ and the complete local ring $\widehat{\mathcal{O}}_{X,p}$ of its singular point. A logical question to ask is whether $\widetilde{X}$ and $\widehat{\mathcal{O}}_{X,p}$ suffice to determine the isomorphism type of $X$. While it may seem at first sight that a singular curve should be completely determined by its smooth locus (encoded by $\widetilde{X}$) and by its singular locus (encoded by the complete local rings of each of its singularities), this is not the case: a curve is more than just the sum of its parts, and the missing piece is precisely the space $\mathcal{R}_{\Gamma}$ which captures how singular points can be ``glued on'' to smooth curves.

\begin{example}
Let $X_1 = \operatorname{Spec} \mathbb{C}[t^2,t^5]$ and let $X_2 = \operatorname{Spec} \mathbb{C}[t^2 + t^3,t^5]$. Both curves have just one singular point at the origin, with complete local rings $\mathcal{O} = \mathbb{C}[[t^2, t^5]]$ and $\mathcal{O'} = \mathbb{C}[[t^2 + t^3, t^5]]$ respectively. We will show that despite having analytically equivalent singularities and isomorphic pointed normalisations, $X_1$ and $X_2$ are not isomorphic. 

Consider the automorphism $\rho$ of $\mathbb{C}[[t]]$ given by $t \mapsto \sqrt{t^2 + t^3} = t + \frac{1}{2} t^2 - \frac{1}{8} t^3 + \cdots$. Then $\rho(\mathcal{O}') = \mathbb{C}[[\rho(t^2), \rho(t^5)]] = \mathbb{C} \left[ \left[  t^2 + t^3, t^5 + \frac{5}{2} t^6 + \cdots \right] \right].$ Since $\rho(\mathcal{O}')$ and $\mathcal{O}$ have semigroup $\langle 2,5 \rangle$ which has conductor $4$, the ideal $(t^{4})$ is contained in both $\rho(\mathcal{O}')$ and $\mathcal{O}$  (see \thref{containsc}). Thus $\rho(t^5) = t^5 + \frac{5}{2} t^6 + \cdots$ is an element of both $\mathcal{O}$ and $\rho(\mathcal{O}')$, and so both rings are equal. Therefore $\mathcal{O}' \cong \mathcal{O}$, that is, $X_1$ and $X_2$ have analytically equivalent singularities. 

The rings $\mathbb{C}[t^2, t^5]$ and $\mathbb{C}[t^2 + t^3, t^5]$ both have $\mathbb{C}(t)$ as their fraction field, and hence $\mathbb{C}[t]$ as their integral closure since $t$ is integral over both rings. Thus $X_1$ and $X_2$ have isomorphic normalisations, corresponding to $\operatorname{Spec} \mathbb{C}[t] = \mathbb{A}^1$. 

Nevertheless, $X_1$ and $X_2$ are not isomorphic. If they were, then there would be an induced isomorphism $\mathbb{C}[t^2, t^5] \rightarrow \mathbb{C}[t^2 + t^3, t^5]$. By the universal property of normalisation, this isomorphism would lift to an isomorphism of their integral closure $\mathbb{C}[t]$ of the form $t \mapsto at$ for some $a \in  \mathbb{C}$. Under such a map, $\mathbb{C}[t^2 , t^5]$ would be sent to $\mathbb{C}[a^2 t^2, a^5 t^5]$, which cannot contain $t^2 + t^3$. Thus $\mathbb{C}[t^2,t^5] \ncong \mathbb{C}[t^2 + t^3, t^5]$ and so $X_1 \ncong X_2$.

\end{example} 

In this example, the two non-isomorphic curves $X_1$ and $X_2$ were obtained by taking two distinct $\mathbb{C}$-subalgebras of $\mathbb{C}[[t]]$ with semigroup $\langle 2,5 \rangle$, which are isomorphic as subrings of $\mathbb{C}[[t]]$, but not isomorphic when viewed as subrings of $\mathbb{C}[t]$. It is the existence of such subalgebras which in general gives rise to a parameter space of ``glueing'', and which leads to the following classification problem.

\begin{Cproblem} \thlabel{Cproblem}
Let $\Gamma \subseteq \mathbb{N}$ be a numerical semigroup and let $Y$ be a smooth curve with a marked point $q \in Y$. Describe the space $$\mathcal{M}_{(Y,q) ,\Gamma} : = \left\{ \begin{aligned} 
&  \text{Morphisms $\pi: (Y,q) \to (X,p)$ where $X$ is a curve with} \\
& \text{a marked point $p \in X$ and $\pi^{-1}(p) = \{q\}$ such that:}\\
& \text{(i) $X$ has normalisation $\pi:Y \rightarrow X$;}\\
& \text{(ii) $p \in X$ is a singular point with semigroup $\Gamma$;} \\
& \text{(iii) $\pi |_{Y \setminus  \{q\} }  : Y \setminus \{q \} \rightarrow X \setminus \{p\}$ is an isomorphism;}
\end{aligned} \right\} \bigg/ \sim,$$ 
where $(X_1,\pi_1) \sim (X_2,\pi_2)$ if and only if there exists an isomorphism $\phi$ from $X_1$ to $X_2$ such that the following diagram  commutes: 
\begin{displaymath}
\begin{tikzcd}
& Y \arrow{dl}[swap]{\pi_1} \arrow{dr}{\pi_2} &  \\
X_1 \arrow{rr}{\phi}[swap]{\cong} & & X_2.
\end{tikzcd}
\end{displaymath}
\end{Cproblem}

Intuitively, the space $\mathcal{M}_{(Y,q), \Gamma}$ consists of the different isomorphism classes of curves that can be obtained by glueing a singularity of a topological type $\Gamma$ onto the curve $Y$ fixing the marked point. However, the equivalence relation on this space is stronger than that of curve isomorphism: by taking the quotient of $\mathcal{M}_{(Y,q), \Gamma}$ by the group of automorphisms of $Y$, we obtain the set of pairs $(X,\pi)$ satisfying conditions (i) to (iii) in the definition of $\mathcal{M}_{(Y,q), \Gamma}$, up to the more familiar equivalence relation of curve isomorphism. 

The advantage of this stronger equivalence relation is that, defined in this way, $\mathcal{M}_{(Y,q),\Gamma}$ has a very simple description:

\begin{prop}
The space $\mathcal{M}_{(Y,q),\Gamma}$ is in bijection with $\mathcal{R}_{\Gamma}$. 
\end{prop} 

\begin{proof} Given a pair $(X,\pi) \in \mathcal{M}_{(Y,q),\Gamma}$, a point in $\mathcal{R}_{\Gamma}$ is obtained simply by taking the complete local ring $\widehat{\mathcal{O}}_{X,p}$ of the singular point of $X$ and identifying it as a subring of $\mathbb{C}[[t]]$ via the injection induced by the normalisation map. The inverse map can be described as follows. Let $\mathcal{O} \subseteq \mathbb{C}[[t]]$ and choose an affine neighbourhood $\operatorname{Spec} A$ of $q \in Y$. Then $\widehat{A}_{\mathfrak{m}_q} \cong \widehat{\mathcal{O}}_{Y,q} \cong \mathbb{C}[[t]]$ and under this identification we have that $\mathbb{C}[[t]]/(t^c) \cong \widehat{A}_{\mathfrak{m}_q} / \mathfrak{m}_q^c \cong A / \mathfrak{m}_{q}^c$. Thus we can view $\mathcal{O}/(t^c)$ as a subring of $A / \mathfrak{m}_{q}^c$. We define the ring $B$ to be the preimage in $A$ of $\mathcal{O}/(t^c)$ under the quotient map $A \to A / \mathfrak{m}_{q}^c$; by construction it is finitely generated and the integral closure of $A$ in its fraction field. The inclusion $B \subseteq A$ induces the normalisation map $\operatorname{Spec} A \to \operatorname{Spec} B$ and we let $p$ denote the image of $q$ under this map. Then $\operatorname{Spec} A \setminus \{p\} \cong \operatorname{Spec} B \setminus \{q\}$. Indeed, choosing a uniformiser $t$ for $\mathfrak{m}_q$ and noting that $t^c \in \mathfrak{m}_q^c \subseteq B$, we see that $B$ and $A$ become equal after localising at $t^c$, where $t^c$ is viewed as a function on $\operatorname{Spec} A$ (respectively $\operatorname{Spec} B$) which vanishes only at $q$ (respectively $p$) to order $c$. A curve $X$ with normalisation $Y$ is then obtained by glueing $\operatorname{Spec} B$ onto $Y \setminus \{q\}$ along the isomorphism $\operatorname{Spec} A \setminus \{q\} \cong \operatorname{Spec} B \setminus \{p\}$. 
\end{proof} 
 
In this way, the space $\mathcal{R}_{\Gamma}$ parametrises the different ways in which a singular point with a given semigroup can be glued on to a given curve $Y$. The space $\mathcal{R}_{\Gamma}$ thus plays an important role in the study of compact moduli spaces of singular curves, as it parametrises certain boundary strata of the moduli spaces. For example, explicit descriptions of the space $\mathcal{R}_{\Gamma}$ for the elliptic $m$-fold point and the ramphoid cusp have permitted the application of intersection theory to the study of the relevant moduli spaces \cite{Smyth2011, Smyth2011a}. The methodology developed here should be useful for carrying out further steps of the Hassett-Keel Program \cite{Smyth2016}.

\section{$\mathcal{R}_{\Gamma}$ is an affine variety}

 \label{affinevariety}
 
In this section we will show that for a numerical semigroup $\Gamma$, $\mathcal{R}_{\Gamma}$ is in bijection with the points of an affine variety. First, we will show that any ring $R \in \mathcal{R}_{\Gamma}$ can be generated by a unique set of $g+1$ polynomials in so-called ``normal'' form with respect to $\Gamma$ (\thref{step1}), where $g+1$ is the size of the set of minimal generators for $\Gamma$. Then, we will identify necessary and sufficient conditions for a set of $g+1$ polynomials in this normal form to generate a ring $R$ with semigroup $\Gamma$ (\thref{step2}). By doing so, we will see that $\mathcal{R}_{\Gamma}$ can be identified with a subset of $\mathbb{C}^{M(\Gamma)}$ defined by the vanishing of a finite number of polynomial functions on $\mathbb{C}^{M(\Gamma)}$ (\thref{step3}). 

We start by providing an alternative but equivalent definition of $\mathcal{R}_{\Gamma}$ which will be more convenient to work with. It relies on the following standard result (cf.\ \cite[Proposition 1.2]{Zariski1986}). 

\begin{prop} \thlabel{containsc}
Let $\Gamma$ be a numerical semigroup with conductor $c$ and suppose that $R \in \mathcal{R}_{\Gamma}$. Then the ideal $(t^c)$ is contained in $R$ and coincides with the conductor ideal of $R$ in $\mathbb{C}[[t]]$, namely the annihilator $\operatorname{Ann}_R(\mathbb{C}[[t]]/R)$. 
\end{prop}

We can now extend the definition of the semigroup of a $\mathbb{C}$-subalgebra of $\mathbb{C}[[t]]$ to $\mathbb{C}$-subalgebras $R$ of $\mathbb{C}[t]/(t^c)$ by defining: $$\Gamma_R : = \{ n \in \mathbb{N} \ | \ \exists \ f \in R \text{ with } \operatorname{ord} f = n\} \cup \{n \in \mathbb{N} \ | \ n \geq c\}.$$ 

By \thref{containsc}, there is a one-to-one correspondence between $\mathbb{C}$-subalgebras of $\mathbb{C}[[t]]$ with semigroup $\Gamma$ and $\mathbb{C}$-subalgebras of $\mathbb{C}[t]/(t^c)$ with semigroup $\Gamma$. The latter perspective is more convenient for our purposes, so we will think of $\mathcal{R}_{\Gamma}$ in the following way: $$\mathcal{R}_{\Gamma} = \{\text{$\mathbb{C}$-subalgebras of $\mathbb{C}[t]/(t^c)$ with semigroup $\Gamma$}\}.$$

\begin{setup}
From hereon, we fix a numerical semigroup $\Gamma = \langle v_0, \hdots, v_g \rangle$, where we assume that $\{v_0,\hdots, v_g \}$ is a set of minimal generators for $\Gamma$ and that $v_i < v_{i+1}$. We let $c$ denote the conductor of $\Gamma$ and denote by $\{n_1, \hdots n_k \}$ the set of elements of $\Gamma$ strictly smaller than $c$, in increasing order (in particular $n_1 = v_0$). An integer $\delta \in \mathbb{N}$ is a \emph{gap} of $\Gamma$ if $\delta \notin \Gamma$ and we denote by $\{\delta_1,\hdots, \delta_l\}$ the set of gaps of $\Gamma$, in increasing order.
 \end{setup}

We start by showing that any ring $R \in \mathcal{R}_{\Gamma}$ can be generated by a unique set $\{x_0(t),\hdots, x_g(t)\}$ in so-called normal form with respect to $\Gamma$. 

\begin{defn} \thlabel{defnormal}
Let $x_0(t),\hdots, x_g(t) \in \mathbb{C}[t]/(t^c)$. The set $\{x_0(t),\hdots,x_g(t)\}$ is in \emph{normal form} (with respect to $\Gamma$) if for each $i$ the following two conditions hold: 
\begin{enumerate}[(i)]
\item $x_i(t)$ is a monic polynomial of order $v_i$; 
\item aside from its leading term\footnote{The leading term of a power series $f \in \mathbb{C}[[t]]$ refers to the monomial with the smallest power.} $t^{v_i}$, the only powers appearing in $x_i(t)$ with non-zero coefficients are gaps of $\Gamma$.  
\end{enumerate}
\end{defn}

\begin{example}
Let $\Gamma = \langle 4,11,14 \rangle$. The gaps of $\Gamma$ are $\{5,6,7,9,10,13,17,21\}$, and so a set $\{x_0(t),x_1(t),x_2(t)\}$ is in normal form if and only if there exists constants $a_i,b_i,c_i \in \mathbb{C}$ such that: \begin{align*} 
x_0(t) & = t^4 + a_5 t^5 + a_6 t^6 + a_7 t^7 + a_9 t^9  + a_{10} t^{10} + a_{13} t^{13} + a_{17} t^{17} + a_{21} t^{21}; \\
x_1(t) & = t^{11} + b_{13} t^{13} + b_{17} t^{17} + b_{21} t^{21}; \\
x_2(t) & = t^{14} + c_{17} t^{17} + c_{21} t^{21}.
\end{align*} 
\end{example}

\begin{prop}\thlabel{step1}
Let $R \in \mathcal{R}_{\Gamma}$. Then there exists a unique set $\{x_0(t),\hdots, x_g(t)\}$ in normal form such that $R = \mathbb{C}[x_0(t),\hdots, x_g(t)]$. 
\end{prop}

\begin{proof}
Since $R$ has semigroup $\Gamma$ there exists a basis $\{y_1(t),\hdots, y_k(t)\}$ for the vector  subspace $R \subseteq \mathbb{C}[[t]]/(t^c)$ such that $y_i(t)$ is monic and of order $n_i$. With respect to this basis for $R$ and to the basis $\{1,t^1, t^2, \cdots t^{c-1}\}$ for $\mathbb{C}[[t]]/(t^c)$, the transpose of the matrix representing to the inclusion of $R$ into $\mathbb{C}[[t]]/(t^c)$ is in row echelon form. Reducing it to its unique reduced row echelon form provides a unique basis $\{y_1'(t),\hdots, y_k'(t)\}$ of $R$ in normal form with respect to $\Gamma$. 

Let $x_i(t) = y_{v_i}'(t)$ for $i = 0 , \hdots, g$. Then $\{x_0(t), \hdots, x_g(t)\}$ is also in normal form with respect to $\Gamma$. We now show that the polynomials $x_i(t)$ generate $R$ as a $\mathbb{C}$-subalgebra. First note that for each $n_i$, there exists a monic polynomial $f_i(t) \in \mathbb{C}[x_0(t),\hdots, x_g(t)]$ of order $n_i$, obtained by choosing a $(g+1)$-tuple $(i_0,\hdots, i_g) \in \mathbb{N}^{g+1}$ satisfying $n_i = \sum_{j=0}^g i_j v_j$. Suppose that $r(t) \in R$ has maximal order $n_k$, and for simplicity assume it is monic. Then it must coincide with $f_k(t)$ since otherwise $r(t) - f_k(t) \in R$ would have order a gap of $\Gamma$. 
Given $r(t) \in R$ monic and of order $n_i$, we have that $r(t) - f_i(t)$ has order strictly greater than $n_i$. By descending induction on the order of elements of $R$, we can assume that $r(t) - f_i(t) \in R$, from which it follows that $r(t) \in R$. 

\end{proof}

\begin{remark} \thlabel{importantrk}
While this proposition implies that $x_0(t),\hdots, x_g(t) $ are generators for $R$, they may not minimally generate $R$. For example, take $\Gamma= \langle 4,6,13 \rangle$, which has conductor $16$, and set $x_0(t) = t^4$, $x_1(t) = t^6 + t^7$ and $x_2(t) = t^{13} - \frac{1}{2} t^{15}$. Let $R = \mathbb{C}[x_0(t),x_1(t),x_2(t)]$. In this case, we have that $x_1(t)^2 - x_0(t)^3 - x_0(t) x_1(t) = 2 t^{13} - t^{15} = 2 x_2(t)$, where we only compute up to order $15$ since $R \subseteq \mathbb{C}[t]/(t^{16})$. Thus $x_2(t) \in \mathbb{C}[x_0(t),x_1(t)]$ and so $R$ is already generated by $x_0(t)$ and $x_1(t)$. 

We will consider this phenomenon in greater generality in Section \ref{stratificatn}, where necessary and sufficient conditions for a ring $R \in \mathcal{R}_{\Gamma}$ to be generated by just two polynomials are determined, in the case when $\Gamma$ has three generators. 
\end{remark}

In light of \thref{step1}, describing $\mathcal{R}_{\Gamma}$ amounts to determining when a set $\{x_0(t),\hdots, x_g(t)\}$ in normal form generates a ring $R = \mathbb{C}[x_0(t),\hdots, x_g(t)]$ with semigroup $\Gamma$. Given such a set, it is clear that $\Gamma$ is contained in $\Gamma_R$. This inclusion may be strict however, that is, $R$ may have a semigroup with a larger set of generators than $\{v_0,\hdots, v_g\}$. For instance, consider the semigroup $\Gamma = \langle 4,6,13 \rangle$ which has conductor $16$, and the set $\{x_0(t),x_1(t), x_2(t)\}$ where $x_0(t) = t^4, x_1(t) = t^6 + t^7$ and $x_2(t) =  t^{13}$. This set is in normal form. Let $R$ denote the generated ring. We have: $$x_1(t)^2 - x_0(t)^3 - 2 x_2(t) - x_0(t)^2 x_1(t) = t^{15}.$$ Thus $15 \in \Gamma_R$, but $15 \notin \Gamma = \langle 4,6,13 \rangle.$ In this case, $\Gamma \subsetneq \Gamma_R$. 

The reason that we were able to obtain an element in $R$ with order lying outside of $\langle 4,6,13 \rangle$ stems from the presence of two distinct polynomials in $R$ with equal order: $x_0(t)^3$ and $x_1(t)^2$, both of order $12$. In general, for $R$ to satisfy $\Gamma_R \subsetneq \Gamma$, there must be some cancellation of elements giving rise to ``new'' orders. To make this precise, we introduce the following notation and definition: 

\begin{notation} \thlabel{notation} 
Given a set $\{x_0(t),\hdots, x_g(t)\}$ in normal form with respect to $\Gamma$, let $R$ denote the ring $\mathbb{C}[x_0(t),\hdots, x_g(t)]$ and $\phi_R$ the induced ring homomorphism $ \mathbb{C}[x_0,\hdots x_g] \rightarrow \mathbb{C}[t]/(t^c)$ defined by $x_i \mapsto x_i(t)$, so that $\phi_R$ is a surjection $\mathbb{C}[x_0,\hdots, x_g] \rightarrow R.$  \end{notation}

\begin{defn} \thlabel{deceptive} 
Let $f \in \mathbb{C}[x_0,\hdots, x_g]$, where $\mathbb{C}[x_0,\hdots, x_g]$ is viewed as a graded ring with $v_i$ the weight of the variable $x_i$. The polynomial $f$ can be written as a finite sum of homogeneous (with respect to the weighted degree) polynomials: $f = f_d + f_{d+1} + \cdots f_{c-1}$, where each $f_i$ is homogeneous of weighted degree $i$ and $f_d \neq 0$. We call $d$ the \emph{weighted order} of $f$. The polynomial $f$ is \emph{deceptive} (with respect to $\Gamma$) if the weighted order of $f$ is strictly larger than the order of $\phi_R(f)$ for any $R$ as in \thref{notation}.  The ideal generated by homogeneous deceptive elements is denoted $I_{dec}(\Gamma)$.  
\end{defn}

\begin{remark} 
Note the following two alternative definitions for a deceptive polynomial: a polynomial $f = f_d + \cdots + f_{c-1} \in \mathbb{C}[x_0,\hdots, x_g]$ with $f_d \neq 0$ is deceptive if and only if the sum of the coefficients of monomials in $f_d$ is equal to zero, or equivalently if and only if $f_d(1,\hdots,1) = 0$. \end{remark}

\begin{example}
Let $\Gamma = \langle 4,6,13 \rangle$. Then $x_0,x_1,x_2$ in $\mathbb{C}[x_0,x_1,x_2]$ have weights $4,6$ and $13$ respectively. The element $x_1^2 - x_0^3$ in $\mathbb{C}[x_0,x_1,x_2]$ is deceptive and homogeneous, and hence lies in $I_{dec}(\Gamma)$. 
\end{example}

Any element of $R$ can be written as the image of a polynomial in $\mathbb{C}[x_0,\hdots, x_g]$ under $\phi_R$. It is clear that if $q \in R$ has as its order a gap of $\Gamma$, then it must be the image of a deceptive polynomial. Nevertheless, a deceptive polynomial in $\mathbb{C}[x_0,\hdots, x_g]$ need not map to an element of $R$ having as its order a gap of $\Gamma$, as shown below:

\begin{example}
Take $\Gamma = \langle 4,6,13 \rangle$ and consider the set $\{x_0(t),x_1(t),x_2(t)\}$ in normal form with $x_0(t) = t^4, x_1(t) = t^6 + t^7 $ and $x_2(t)= t^{13}$. Let $R = \mathbb{C}[x_0(t),x_1(t),x_2(t)]$. The polynomial $x_1^2 - x_0^3$ is deceptive with respect to $\Gamma$, and we have: $\phi_R(x_1^2 - x_0^3) = x_1(t)^2 - x_0(t)^3 = 2 t^{13} + t^{14}$, which has as its order an element of $\Gamma$. 
\end{example} 

However, if the image of a deceptive polynomial under $\phi_R$ \emph{does} have as its order an element of $\Gamma$, then we can successively remove from it all powers lying in $\Gamma$: 

\begin{example}
In the previous example, we saw that $\phi_R(x_1^2 - x_0^3) = 2 t^{13 } + t^{14}$. Since $13 \in \Gamma$, we can remove this power by subtracting an appropriate element of $R$, in this case $x_2(t)$: $$\phi_R(x_1^2 - x_0^3) - 2 x_2(t) = t^{14}.$$ Again, $t^{14}$ lies in $\Gamma$ so it can be removed by subtracting an element in $R$ of order $14$, in this case $x_0 (t)^2 x_1(t)$: 
$$\phi_R(x_1^2 - x_0^3) - 2 x_2(t) - x_0(t)^2 x_2(t) = t^{15},$$ which has as its order a gap of $\Gamma$, since $15 \notin \Gamma$. Thus while the image of the deceptive element $x_1^2 - x_0^3$ has as its order an element of $\Gamma$, after subtracting those powers lying in $\Gamma$, we obtain an element still in $R$, with order a gap of $\Gamma$. \end{example}

The process of removing powers lying in $\Gamma$ using elements of $R$ illustrated in the above example is completely algorithmic and can be used to define a `reduction' map for elements of $R$ (see \thref{reddefin}). As we will see in \thref{step2}, the condition then for a ring $R$ to have semigroup $\Gamma$ is that the image of deceptive elements of $\mathbb{C}[x_0,\hdots, x_g]$ in $R$ is mapped to zero under the reduction map.

\begin{algorithm} \thlabel{redalgorithm}

Let $\{x_0(t),\hdots, x_g(t)\}$ be in normal form and let $R = \mathbb{C}[x_0(t),\hdots, x_g(t)]$. Recall, as in the proof of \thref{step1}, that for each $n_i$ there exists a monic polynomial $f_i(t) \in R$ of order $n_i$, corresponding to a solution $(i_0,\hdots,i_g) \in \mathbb{N}^{g+1}$ of the equation $\sum_{j=0}^{g} i_j v_j = n_i$. To make the choice of $f_i(t)$ unique, we require that $f_i(t)$ corresponds to the smallest solution $(i_0,\hdots, i_g)$ with respect to reverse lexicographic ordering. Note that if $(i_0,\hdots, i_g)$ is such a solution, then $f_i(t) = \phi_R(F_i)$ where $F_i = \prod_{j=0}^g x_0^{i_0} \cdots x_g^{i_g} \in \mathbb{C}[x_0, \hdots, x_g]$.

Let $r(t)  \in \mathbb{C}[t]/(t^c)$. The algorithm consists in successively removing, in increasing order, each power $n_i$ appearing in $r(t)$ with a non-zero coefficient by subtracting the appropriate multiple of $f_i(t)$. This produces a uniquely determined polynomial in $R$ satisfying the property that each power appearing with a non-zero coefficient is a gap of $\Gamma$.

\end{algorithm} 

\begin{defn} \thlabel{reddefin}
Given $R \in \mathcal{R}_{\Gamma}$ and for each $n_i$ a corresponding $f_i(t) \in R$ of order $n_i$ defined as in \thref{redalgorithm}, the \emph{reduction map associated to $R$} is the map $\operatorname{red}_{R}: R \to R$ mapping a polynomial $r(t) \in R$ to the polynomial obtained by applying \thref{redalgorithm} to it. The polynomial $\operatorname{red}_R(r(t))$ is called the \emph{reduced form of $r(t)$ with respect to $R$}. 
\end{defn} 

We record the following properties of the map $\operatorname{red}_R$, which will be used to prove  \thref{step2}. 

\begin{prop} \thlabel{rkalgorithm} Given $r(t) \in R$, the polynomial $\operatorname{red}_R(r(t))$ satisfies the following properties: 
\begin{enumerate}[(i)]
\item There exists a polynomial $F_{r(t)}$ in $\mathbb{C}[x_0,\hdots, x_g]$, determined by $r(t)$, such that $\operatorname{red}_R(r(t)) = r(t) - \phi_R(F_{r(t)})$ and moreover such that the weighted order of $F_{r(t)}$ is strictly larger than the order of $r(t)$;
\item The coefficients of $r(t)$ are polynomials in the coefficients of $x_0(t), \hdots, x_g(t)$.
\end{enumerate}
\end{prop} 

\begin{proof} 
For the first part, we note that by definition of $\operatorname{red}_R(r(t))$ in terms of \thref{redalgorithm}, the polynomial $\operatorname{red}_R(r(t))$ is obtained by removing from $r(t)$ a linear combination of the polynomials $f_i(t)$, say $\sum_{j=0}^m a_j f_j(t)$. By definition $f_i(t) = \phi_R(F_i)$ where $F_i \in \mathbb{C}[x_0, \hdots, x_g]$ of weighted order $n_i$ is determined by a solution to the equation $\sum_{j=0}^g i_j v_j = n_i$. Thus $\sum_{j=0}^m a_j f_j(t) =  \phi_R(\sum_{j=0}^m a_j F_j)$ and so $\operatorname{red}_R(r(t)) = r(t) - \phi_R(F_{r(t)})$ where $F_{r(t)} = \sum_{j=0}^m a_j F_j$. Moreover, the lowest weighted order of any $F_j$ appearing with non-zero coefficient in $F_{r(t)}$ is strictly greater than the order of $r(t)$, since only powers strictly greater than the order of $r(t)$ are removed when applying \thref{redalgorithm}.

For the second part, since $R = \mathbb{C}[x_0(t), \hdots, x_g(t)]$, the polynomial $r(t)$ can be written as a polynomial in $x_0(t), \hdots, x_g(t)$. Thus its coefficients are polynomials in the coefficients of the generators $x_i(t)$. Moreover, the coefficients of the above polynomial $\sum_{j=0}^m a_j f_j(t)$ are also polynomials in the coefficients of the generators $x_i(t)$ since $f_i(t) \in R$ for each $i$. The result then follows from the first part, since $\operatorname{red}_R(r(t)) =  r(t) - \sum_{j=0}^m a_j f_j(t)$.

\end{proof}

We now introduce a slight generalisation of \thref{redalgorithm} which will be needed in Section \ref{stratificatn} when identifying the subset of $\mathcal{R}_{\Gamma}$ consisting of plane curve singularities. 

\begin{algorithm}[Generalisation of \thref{redalgorithm}]
Let $\Gamma = \langle v_0, \hdots, v_g \rangle$ and let $\{ v_{i_1}, \hdots, v_{i_k}\}$ be an ordered subset of $\{v_1, \hdots, v_g\}$. Let $R = \mathbb{C}[x_0(t),\hdots, x_g(t)]$ where $\{x_0(t), \hdots, x_g(t)\}$ is in normal form with respect to $\Gamma$. Given an element $r(t) \in \mathbb{C}[t]/(t^c)$, we construct a new polynomial denoted $\operatorname{red}_{\langle v_0, \hdots, v_g \rangle}(r(t))$ as follows. Instead of successively removing all powers $n < c$ lying in $\Gamma = \langle v_0, \hdots, v_g \rangle$ as in the definition of $\operatorname{red}_R(r(t))$, we only remove those powers lying in $\langle v_{i_1}, \hdots, v_{i_k} \rangle$. In this way, we only use the polynomials $x_{i_1}(t), \hdots, x_{i_k}(t)$ to remove powers from $f(t)$. 
\end{algorithm}  

\begin{remark}
By definition $\operatorname{red}_{\langle v_0, \hdots, v_g \rangle} = \operatorname{red}_R$ as linear maps from $R$ to itself. Moreover, the properties of $\operatorname{red}_R$ given in \thref{rkalgorithm} also hold for $\operatorname{red}_{\langle v_{i_1}, \hdots, v_{i_k} \rangle}$, with the difference that the polynomial $F_{r(t)}$ lies in $\mathbb{C}[x_{i_1}, \hdots, x_{i_k}]$ rather than $\mathbb{C}[x_0, \hdots, x_g]$, and that the coefficients of $\operatorname{red}_{\langle v_0, \hdots, v_g \rangle}(r(t))$ are polynomials in the coefficients of $x_{i_1}(t), \hdots, x_{i_k}(t)$ only, instead of $x_0(t), \hdots, x_g(t)$.   
\end{remark}

With the reduction map in hand, we can now precisely formulate when a set $\{x_0(t),\hdots, x_g(t)\}$ in normal form generates a ring with semigroup $\Gamma$. 

\begin{prop} \thlabel{step2}
Let $\{f_i\}_{i \in S}$ be generators of $I_{dec}(\Gamma) \subseteq \mathbb{C}[x_0,\hdots, x_g]$\footnote{Since $\mathbb{C}[x_0,\hdots, x_g]$ is noetherian, a finite generating set exists but the finiteness hypothesis is not needed here.}. Then the following are equivalent: \begin{enumerate}[(i)]
\item $R \in \mathcal{R}_{\Gamma}$;
\item $\operatorname{red}_R (\phi_R(f)) = 0$ for all deceptive $f \in \mathbb{C}[x_0,\hdots, x_g]$;
\item $\operatorname{red}_R (\phi_R(f_i)) = 0$ for all $i \in S$. 
\end{enumerate}
\end{prop} 

\begin{proof}
(i) $\Rightarrow$ (ii) Suppose that $R$ has semigroup $\Gamma$, and let $f$ be deceptive. Then by definition of the map $\operatorname{red}_R$, the only non-zero powers of $\operatorname{red}_R (\phi_R(f))$ are gaps of $\Gamma$. Thus $\operatorname{red}_R (\phi_R(f))$ must be zero to ensure that $\Gamma_R = \Gamma$. \\

\noindent (ii) $\Rightarrow$ (iii) This follows immediately from the fact that $f_i \in I_{dec}(\Gamma)$ for all $i \in S$ and hence is deceptive. \\

\noindent (iii) $\Rightarrow$ (i)  Suppose that $\operatorname{red}_R (\phi_R(f_i)) = 0$ for all $i \in S$. Suppose, in order to reach a contradiction, that there exists some $q \in R$ with $\operatorname{ord} q = n \notin \Gamma$. We can choose a preimage $\overline{q} \in \mathbb{C}[x_0,\hdots, x_g]$ of $q$ under $\phi_R$ which has maximal weighted order. Such a $\overline{q}$ exists, since the weighted order of any preimage of $q$ is bounded above by $n$. We write $\overline{q} = q_d + q_{d+1} + \cdots$ where each $q_i$ is weighted homogeneous of degree $d_i$. The polynomial $\overline{q}$ is deceptive, since otherwise the order of $q$ would lie in $\Gamma$. Now $q_d$ is a homogeneous deceptive element, therefore $q_d \in I_{dec}(\Gamma)$. Thus there exists elements $r_1,\hdots, r_n \in R$ and $f_1, \hdots, f_n \in \{f_i\}_{i \in S}$ such that $q_d = \sum_{i=1}^n r_i f_i$.  By \thref{rkalgorithm}, there exists for each $i$ a polynomial $F_{\phi_R(f_i)} \in \mathbb{C}[x_0,\hdots, x_g]$ of weighted order strictly larger than that of $f_i$ such that $\operatorname{red}_R(\phi_R(f_i)) = \phi_R(f_i) - \phi_R(F_{\phi_R(f_i)})$. To simplify notation we set $F_i = F_{\phi_R(f_i)}$. Since by assumption $\operatorname{red}_R( \phi_R(f_i)) = 0$ for all $i \in S$, it follows that $\phi_R(f_i) = \phi_R(F_i)$ for each $i$.

 Let $Q = \sum_{i=1}^n r_i F_i$, and let $\overline{Q} = Q + q_{d+1} + \cdots$ so that $\overline{Q}$ has weighted order strictly greater than that of $\overline{q}$ by assumption on the $F_i$. Then by construction we have: $\phi_R(\overline{Q}) = \phi_R(\overline{q}) = q$. This contradicts the choice of $\overline{q}$ as a preimage with maximal weighted order. Thus there can be no element $q \in R$ with order lying outside of $\Gamma$. 
\end{proof} 

Hence to determine whether a set $\{x_0(t),\hdots, x_g(t)\}$ in normal form generates a ring with semigroup $\Gamma$, it suffices to check that $\operatorname{red}_R (\phi_R(f_i)) = 0$ for any set of generators $\{f_i\}_{i \in S}$ of $I_{dec}(\Gamma)$. We now explicitly identify a generating set of elements for $I_{dec}(\Gamma)$.

\begin{defn} \thlabel{Sdec} Let $\Gamma = \langle v_0, \hdots, v_g \rangle$ be a numerical semigroup with conductor $c$. We define the subset $S_{dec}(\Gamma)$ of $I_{dec}(\Gamma) \subseteq \mathbb{C}[x_0,\hdots, x_g]$ by:
 $$S_{dec}(\Gamma) := \left\{ x_0^{i_0} \cdots x_g^{i_g} - x_0^{i_0'} \cdots x_g^{i_g'} \ \left|  \ \begin{aligned}
& \text{(i)}  \  i_j, i_j' \in \mathbb{N} \text{ and } (i_0,\hdots i_g) \neq (i_0',\hdots, i_g')\\
& \text{(ii)} \  \text{$i_m < i_m'$ where $m$ is the smallest $j$ such that $i_j \neq i_j'$}\\
& \text{(iii)} \    \sum_{j=0}^g i_j v_j = \sum_{j=0}^{g} i_j' v_j  
\end{aligned} \right. \right\}.$$

  \end{defn}  
 
Condition (iii) ensures that polynomials in $S_{dec}(\Gamma)$ are deceptive, while conditions (i) and (ii) simply ensure that $S_{dec}(\Gamma)$ does not contain both a polynomial and its negative.

\begin{example}
Let $\Gamma = \langle 3,5 \rangle$, and let $\{x_0(t),x_1(t)\}$ be in normal form. Then $x_0(t)$ has order $3$ and $x_1(t)$ has order $5$, so $x_1^3 - x_0^5 \in \mathbb{C}[x_0,x_1]$ is an element of $S_{dec}(\Gamma)$, as is $x_1^6 - x_0^{10}$ and so on. 
\end{example}
 
  \begin{prop} \thlabel{generate}
The polynomials in $S_{dec}(\Gamma)$ generate the ideal $I_{dec}$.
  \end{prop} 
  
 \begin{proof}
Since by definition $I_{dec}(\Gamma)$ is the ideal generated by the deceptive homogeneous polynomials of $\mathbb{C}[x_0,\hdots, x_g]$, it suffices to show that any deceptive homogeneous polynomial $f$ lies in the ideal generated by the elements of $S_{dec}(\Gamma)$, denoted $(g)_{g \in S_{dec}(\Gamma)}$. The polynomial $f$ can be written as a sum of complex multiples of monic monomials $f_1,\hdots, f_n$: $f = \sum_{i=1}^n a_i f_i$ for some $a_i \in \mathbb{C}$. Since $f$ is deceptive, we must have that $\sum_{i=1}^n a_i = 0$, so that $a_1 = - (a_2 + \cdots + a_n)$. Thus $f = a_2 (f_2 - f_1) + a_3 (f_3 - f_1) + \cdots + a_n (f_n - f_1)$. Each $f_i - f_1$ is, up to multiplication by $-1$, an element of $S_{dec}(\Gamma)$. Thus $f$ lies in $(g)_{g \in S_{dec}(\Gamma)}$. It follows that elements of $S_{dec}(\Gamma)$ generate the ideal $I_{dec}(\Gamma)$. 
 \end{proof} 
 
Applying \thref{step2}, we can conclude that $R= \mathbb{C}[x_0(t),\hdots, x_g(t)]$, with $\{x_0(t),\hdots, x_g(t)\}$ in normal form, lies in $\mathcal{R}_{\Gamma}$ if and only if $\operatorname{red}_R( \phi_R(f)) = 0$ for all $f \in S_{dec}(\Gamma)$. Note that if $f \in S_{dec}(\Gamma)$ has weighted degree larger than the conductor $c$ of $\Gamma$, then $\operatorname{red}_R(\phi_R(f))$ is automatically zero. Thus to ensure that $R$ has semigroup $\Gamma$, it suffices to check that $\operatorname{red}_R(\phi_R(f)) = 0$ for those $f \in S_{dec}(\Gamma)$ with weighted degree less than or equal to $c$.

We have thus obtained an algorithmic procedure to determine whether or not a given set in normal form generates a ring with semigroup $\Gamma$: it suffices to check whether the reduced forms of a finite number of polynomials is zero or not. This can be done by a computer, and the following example illustrates the procedure.

  \begin{example} \thlabel{4613Sdec}
 Let $\Gamma = \langle 4,6,13 \rangle$. The conductor of $\Gamma$ is $16$. Suppose that a relation of the form $4 i  + 6j + 13 k =0$ holds for some $(i,j,k) \in \mathbb{Z}^3$. Suppose first that $i \leq 0$ and that $j, k \geq 0$, so that we have $4 (-i) = 6 j + 13k$ with both sides of the equation positive. The only such relation to hold below the conductor is obtained by taking $i = -3$, $j = 2$ and $k=0$. Thus $y^2 - x^3 \in S_{dec}(\Gamma) \subseteq \mathbb{C}[x,y,z]$. Now suppose that $j \leq 0$ and that $i,k \geq 0$. The equation $6(-j) = 4 i + 13k$ is satisfied below the conductor only when $j = -2, i = 3$ and $k = 0$, which yields the same polynomial as above. Finally, there is no solution to the equation $13 (-k) = 4 i + 6j$ with $k \leq 0$ and $i,j \geq 0$ which holds below the conductor. Thus the only polynomial in $S_{dec}(\Gamma)$ with order smaller than $16$ is $y^2 - x^3$. 

Given a triple $\{x(t),y(t),z(t)\}$ in normal form, to determine whether or not the generated ring $R$ has semigroup $\Gamma$, it suffices therefore to check whether $\operatorname{red}_R(\phi_R(y^2 - x^3))$ is zero or not.

Since the conductor of $\Gamma$ is $c = 16$, any set $\{x(t), y(t), z(t)\}$ in normal form with respect to $\Gamma$ satisfies: \begin{align*}
x(t) & = t^4 + a_5 t^5 + a_7 t^7 + a_9 t^9 + a_{11}t^{11} +a_{15} t^{15}; \\
y(t) & = t^6 + b_7 t^7 + b_9 t^9 + b_{11} t^{11} + b_{15} t^{15}; \\
z(t) & = t^{13} + c_{15} t^{15}.
\end{align*} for some $a_i,b_i, c_i \in \mathbb{C}$. Setting $R = \mathbb{C}[x(t),y(t),z(t)]$, we can compute the following: $$\operatorname{red}_R (\phi_R(y^2 - x^3)) = (5 a_5^3 + 3 a_5^2 b_7 - 2 a_5 b_7^2 + 3 a_5 c_{15} - 3 a_7 - b_7^3 - 2 b_7 c_{15} + 2 b_9) t^{15}.$$ 

Thus $R \in \mathcal{R}_{\Gamma}$ if and only if $5 a_5^3 + 3 a_5^2 b_7 - 2 a_5 b_7^2 + 3 a_5 c_{15} - 3 a_7 - b_7^3 - 2 b_7 c_{15} + 2 b_9 = 0$. 
 \end{example}

The main result of this section is now just a simple corollary of \thref{generate}. Note that this set-theoretic result overlaps with the scheme-theoretic Theorem 3 of \cite{Ishii1980} which shows that the moduli functor associated with this classification problem is representable by an affine scheme. 

\begin{thm} \thlabel{step3}
Let $f_1,\hdots, f_n$ denote the polynomials of $S_{dec}(\Gamma)$ of weighted degree less than the conductor of $\Gamma$, and let $f_1^{(i)},\hdots, f_{r_i}^{(i)}$ denote the coefficients of $\operatorname{red}_R(\phi_R(f_i))$ for $i = 1,\hdots, n$. Moreover, let $$M(\Gamma) = \sum_{i=0}^g \text{ \# gaps of $\Gamma$ greater than $v_i$}.$$ Then $\mathcal{R}_{\Gamma}$ is in bijection with the $\mathbb{C}$-points of the affine subvariety $$V(f_1^{(1)},\hdots, f_{r_1}^{(1)},\hdots, f_{1}^{(n)},\hdots,f_{r_n}^{(n)}) \subseteq \mathbb{C}^{M(\Gamma)}.$$  
\end{thm}

\begin{proof}
By \thref{step2,generate}, we have: $$\mathcal{R}_{\Gamma} \leftrightarrow \{ \{x_0(t),\hdots, x_g(t)\} \text{ in normal form} \ | \ \operatorname{red}_R (\phi_R(f)) = 0 \text{ for all } f \in S_{dec}(\Gamma) \text{ with } \operatorname{ord} f < c\}.$$ The set of all sets $\{x_0(t),\hdots, x_g(t)\}$ in normal form can naturally be identified with $\mathbb{C}^{M(\Gamma)}$ for $M(\Gamma) = \sum_{i=0}^g \text{ \# gaps of $\Gamma$ greater than $v_i$}.$ 
Thus $\mathcal{R}_{\Gamma}$ is a subset of $\mathbb{C}^{M(\Gamma)}$, consisting of those $M(\Gamma)$-tuples such that the corresponding polynomials $x_0(t),\hdots, x_g(t)$ satisfy $\operatorname{red}_R (\phi_R(f)) = 0$ for all $f \in S_{dec}(\Gamma)$, where $R = \mathbb{C}[x_0(t),\hdots, x_g(t)]$ and $\phi_R: \mathbb{C}[x_0,\hdots, x_g] \rightarrow \mathbb{C}[t]/(t^c)$ denotes the induced ring homomorphism.

By \thref{rkalgorithm}, the coefficients $f_1^{(i)}, \hdots, f_{r_i}^{(i)}$ of each $\operatorname{red}_R (\phi_R(f_i))$ are polynomials in the coefficients of $x_0(t),\hdots, x_g(t)$. Hence each $f_j^{(i)}$ can be viewed as a function on $\mathbb{C}^M$. The set $\mathcal{R}_{\Gamma}$ can therefore be identified with the subset of $\mathbb{C}^{M(\Gamma)}$ consisting of the vanishing locus of the functions $f_1^{(1)},\hdots, f_{r_1}^{(1)},\hdots,f_{1}^{(n)},\hdots, f_{r_n}^{(n)}$. 
\end{proof}

\section{Computing $\mathcal{R}_{\Gamma}$: examples} \label{manyexamples}

In this section we explicitly determine the affine variety corresponding to $\mathcal{R}_{\Gamma}$ for various semigroups $\Gamma$, by computing the polynomials $f_1^{(1)},\hdots, f_{r_1}^{(1)},\hdots,f_{1}^{(n)},\hdots, f_{r_n}^{(n)}$.

\begin{example}[$\Gamma = \langle 3,5 \rangle$]
The conductor of $\Gamma$ is $c = 8$. Thus any set $\{x(t),y(t)\}$ in normal form satisfies $x(t) = t^3 + a_4 t^4 + a_7 t^7$ and $y(t) = t^5 + b_7 t^7$ for some $a_i,b_i \in \mathbb{C}$. An element of $S_{dec}(\Gamma)$ of weighted degree smaller than the conductor is of the form $x^i y^j - x^{i'} y^{j'}$ for some $i,j,i',j' \in \mathbb{N}$, with $3i + 5j = 3 i' + 5j' < 8$. The latter equation can be rearranged so that it is of the form $3i = 5j$ for some $(i,j) \in \mathbb{N}^2$. The smallest integer at which $3 i = 5j$ is the least common multiple of $3$ and $5$, equal to $15$, which is strictly larger than the conductor. Thus no relation between $3$ and $5$ holds below the conductor, and so $S_{dec}(\Gamma)$ is empty. 

Using the notation from \thref{step3}, we have that $$M = \text{\# gaps of $\Gamma$ larger than $3$} + \text{\# gaps of $\Gamma$ larger than $5$} = 2 + 1 = 3.$$ Thus from \thref{step3}, we have that \begin{align*} 
\mathcal{R}_{\Gamma} & \leftrightarrow V(\emptyset) \subseteq \mathbb{C}^3 \\
& \leftrightarrow \mathbb{C}^3, 
\end{align*} 
with the map $\leftarrow$ given by $$(a,b,c) \mapsto \mathbb{C}[t^3 + a t^4 + b t^7, t^5 + c t^7] \subseteq \mathbb{C}[t]/(t^8).$$
\end{example} 

This result can be generalised to an arbitrary numerical semigroup with two generators, a generalisation which overlaps with \cite[Corollary 5]{Ishii1980}.

\begin{thm} \thlabel{2gen} 
Let $\Gamma = \langle v_0,v_1 \rangle$ be a numerical semigroup. Then $\mathcal{R}_{\Gamma} \leftrightarrow \mathbb{C}^{M(\Gamma)}$ where $$M(\Gamma) = \text{\# gaps of $\Gamma$ greater than $v_0$} + \text{\# gaps of $\Gamma$ greater than $v_1$}.$$
\end{thm} 

\begin{proof}
The conductor of $\Gamma$ is $c = (v_0 - 1)(v_1 - 1)$. The smallest relation involving $v_0$ and $v_1$ occurs at the least common multiple of $v_0$ and $v_1$, equal to $v_0 v_1$, which is strictly larger than $c$. Thus $S_{dec}(\Gamma)$ contains no polynomial of weighted degree smaller than  the conductor, and so \begin{align*} 
\mathcal{R}_{\Gamma} & \leftrightarrow V(\emptyset) \subseteq \mathbb{C}^{M(\Gamma)} \\
& \leftrightarrow \mathbb{C}^{M(\Gamma)}
\end{align*} by \thref{step3}. 
\end{proof} 

We now compute $\mathcal{R}_{\Gamma}$ for a semigroup with three generators. 

\begin{example}[$\Gamma = \langle 4,6,13 \rangle$] \thlabel{4613}

We have: $\sum_{i=0}^2 \# \text{gaps of $\Gamma$ greater than $v_i$} = 5 + 4 + 1 = 10.$ As seen in \thref{4613Sdec}, $$R = \mathbb{C}[t^4 + a_5 t^5 + a_7 t^7 + a_9 t^9 + a_{11}t^{11} +a_{15} t^{15}, t^6 + b_7 t^7 + b_9 t^9 + b_{11} t^{11} + b_{15} t^{15}, t^{13} + c_{15} t^{15}]$$ has semigroup $\Gamma$ if and only if $5 a_5^3 + 3 a_5^2 b_7 - 2 a_5 b_7^2 + 3 a_5 c_{15} - 3 a_7 - b_7^3 - 2 b_7 c_{15} + 2 b_9 = 0$. That is, $$\mathcal{R}_{\Gamma} \leftrightarrow V(5 a_5^3 + 3 a_5^2 b_7 - 2 a_5 b_7^2 + 3 a_5 c_{15} - 3 a_7 - b_7^3 - 2 b_7 c_{15} + 2 b_9) \subseteq \mathbb{C}^{10},$$ where $\mathbb{C}^{10}$ is given coordinates $a_5, a_7, a_9, a_{10}, a_{11}, a_{15}, b_7, b_9, b_{11}, b_{15}, c_{15}$. Consider the projection map $$p:V(5 a_5^3 + 3 a_5^2 b_7 - 2 a_5 b_7^2 + 3 a_5 c_{15} - 3 a_7 - b_7^3 - 2 b_7 c_{15} + 2 b_9)  \rightarrow \mathbb{C}^9$$ obtained by omitting the coordinate $b_9$. This map is invertible: its inverse is obtained by setting $$b_9 = - \frac{1}{2} \left(5 a_5^3 + 3 a_5^2 b_7 - 2 a_5 b_7^2 + 3 a_5 c_{15} - 3 a_7 - b_7^3 - 2 b_7 c_{15} \right).$$ Thus $$\mathcal{R}_{\Gamma} \leftrightarrow \mathbb{C}^9.$$

\end{example}

In the above example, there was only one polynomial in $S_{dec}(\Gamma)$ of weighted degree less than the conductor. We now consider an example of a three generator semigroup for which $S_{dec}(\Gamma)$ has two polynomials of weighted degree less than the conductor. 
 
\begin{example}[$\Gamma = \langle 9,16,19 \rangle$] \thlabel{91619}
The conductor of $\Gamma$ is $c = 59$. The gaps of $\Gamma$ are: $$\mathbb{N} \setminus \Gamma = \{1,2,3,4,5,6,7,8,10,11,12,13,14,15,17,20,21,22,23,24,26,29,30,31,33,39,40,42,49,58\}.$$ Any set $\{x(t),y(t),z(t)\}$ in normal form satisfies: \begin{align*}
x(t) & = t^9 + a_{10} t^{10} + a_{11} t^{11} + a_{12} t^{12} + a_{13} t^{13} + a_{14} t^{14} + a_{15} t^{15} + a_{17} t^{17} + a_{20} t^{20}  + a_{21} t^{21} + \cdots + a_{58} t^{58}; \\
y(t) & = t^{16} + b_{17} t^{17} + b_{20} t^{20} + b_{21} t^{21} + b_{22} t^{22} + b_{23} t^{23} + b_{24} t^{24} + b_{26} t^{26} + b_{29} t^{29} + b_{30} t^{30} + \cdots + b_{58} t^{58}; \\
 z(t) & = t^{19} + c_{20} t^{20} + c_{21} t^{21} + c_{22} t^{22} + c_{23} t^{23} + c_{24} t^{24} + c_{26} t^{26} + c_{29} t^{29} + c_{30} t^{30} + c_{31} t^{31} + \cdots + c_{58} t^{58}
\end{align*} for some $a_i, b_i, c_i \in \mathbb{C}$. 
We now must determine the elements of $S_{dec}(\Gamma) \subseteq  \mathbb{C}[x,y,z]$. These are determined by solutions $(i,j,k,i',j',k') \in \mathbb{N}^6$ to the equation $9i + 16 j + 19 k = 9i' + 16 j' + 19 k'$ satisfying $9i + 16 j + 19 k < 59$. Any such solution can be obtained from a solution of $(i,j,k) \in \mathbb{N}^3$ to an equation in $\mathbb{N}[i,j,k]$ of one of the following three forms: (i) $9i = 16 j + 19 k < 59$, (ii) $16 j = 9 i + 19 k < 59$ and (iii) $29 k = 9 i + 16 j < 59$. Straightforward computation shows that equation (ii) does not admit a solution below $c=59$, and that (i) and (iii) admit a unique solution below the conductor: (i) $9 \times 6 = 16 + 19 \times 2$ and (iii) $19 \times 3 = 9 + 16 \times 3$. 
Thus $S_{dec}(\Gamma)$ consists of exactly two polynomials: $$S_{dec}(\Gamma) = \{ x^6 - y z^2, z^3 - x y^3\}.$$ Let $f := x^6 - y z^2$ and let $g := z^3 - x y^3$. Using a programme devised in Python which computes the reduced form of polynomials with respect to a given semigroup, we obtain: \begin{align*}
\operatorname{red}_R (\phi_R(f)) & = (-2 a_{10}^4 - 4 a_{10}^3 b_{17} + 4 a_{10}^3 c_{20} + 6 a_{10}^2 b_{17} + 6 a_{10}^2 b_{17}^2 - 13 a_{10}^2 b_{17} c_{20} + 10 a_{10}^2 c_{20}^2 - 2 a_{10}^2 c_{21} \\
& \ \ \ \  \ + 18 a_{10} a_{11} b_{17} - 24 a_{10} a_{11} c_{20} + 6 a_{10} b_{17}^2 c_{20} - 7 a_{10} b_{17} c_{20}^2 + 3 a_{10} b_{17} c_{21} - 4 a_{10} c_{20}^3\\
& \ \ \ \ \ + 8 a_{10} c_{20} c_{21}  - 4 a_{10} c_{22} + 3 a_{11}^2 - 12 a_{11} b_{17}^2 + 18 a_{11} b_{17} c_{20} - 8 a_{11} c_{21} - 14 a_{12} b_{17} \\
& \ \ \ \ \ + 8 a_{12} c_{20}  + 6 a_{13} - 3 b_{17}^2 c_{20}^2  + 3 b_{17}^2 c_{21} + 7 b_{17} c_{20}^3 - 12 b_{17} c_{20} c_{21} + 5 b_{17} c_{33}\\
& \ \ \ \ \  - b_{20} - c_{20}^4 + 3 c_{21}^2 - 2 c_{23}) t^{58}; \\
\operatorname{red}_R(\phi_R(g)) & = (-a_{10} - 3 b_{17} + 3 c_{20}) t^{58}. 
\end{align*} 

Let $g_1$ denote the coefficient in front of $t^{58}$ in $\operatorname{red}_R(\phi_R(f))$ and let $g_2$ denote the coefficient in front of $t^{58}$ in $\operatorname{red}_R (\phi_R(g))$. By \thref{step3}, $\mathcal{R}_{\Gamma} \leftrightarrow V(g_1,g_2) \subseteq \mathbb{C}^{53}$, since $53 = \sum_{i=0}^2 \text{\# gaps of $\Gamma$ greater than $v_i$}.$ Since $g_1$ and $g_2$ depend linearly on $a_{13}$ and $c_{20}$ respectively, and moreover since the linear parts of $g_1$ and $g_2$ are linearly independent, the projection map $\mathbb{C}^{53} \to \mathbb{C}^{51}$ obtained by omitting the variables $a_{13}$ and $c_{20}$ is invertible. Thus $$\mathcal{R}_{\Gamma} \leftrightarrow \mathbb{C}^{51}.$$ 
  
\end{example}

We now consider a semigroup with four generators. 

\begin{example}[$\Gamma = \langle 8,9,10,11 \rangle$] \thlabel{891011}
The conductor of $\Gamma$ is $c = 24$ and its gaps are: $$\mathbb{N} \setminus \Gamma = \{1,2,3,4,5,6,7,12,13,14,15,23 \}.$$ Any set $\{x(t),y(t),z(t),w(t)\}$ in normal form satisfies: \begin{align*} 
x(t) & = t^8 + a_{12} t^{12} + a_{13} t^{13} + a_{14} t^{14} + a_{15} t^{15} + a_{23} t^{23}; \\
y(t) & = t^9 + b_{12} t^{12} + b_{13} t^{13}+ b_{14} t^{14} + b_{15} t^{15} + b_{23} t^{23}; \\
z(t) & = t^{10} + c_{12} t^{12} + c_{13} t^{13} + c_{14} t^{14} + c_{15} t^{15} + c_{23} t^{23}; \\
w(t) & = t^{11} + d_{12} t^{12} + d_{13} t^{13} + d_{14} t^{14} + d_{15} t^{15} + d_{23} t^{23},
\end{align*} 
for some $a_i, b_i, c_i, d_i \in \mathbb{C}$. Elements of $S_{dec}(\Gamma) \subseteq \mathbb{C}[x,y,z,w]$ of weighted degree less than $c$ are determined by solutions $(i,j,k,l,i',j',k',l') \in \mathbb{N}^{8}$ to the equation $8 i + 9 j + 10 k + 11 l = 8 i' + 9 j' + 10 k' + 11 l'$ satisfying $8 i + 9 j + 10 k + 11 l < 24$. There are three such solutions: \begin{align*} 
\text{(i) } & 9 \times 2  = 8 \times 1 + 10 \times 1; \\
\text{(ii) } & 8 \times 1 + 11 \times 1   = 9 \times 1 + 10 \times 1; \\
\text{(iii) } & 10 \times 2  = 9 \times 1 + 11 \times 1. 
\end{align*} 
Thus $S_{dec}(\Gamma)$ consists of three polynomials: $$S_{dec}(\Gamma)  = \{ y^2 - xz, xw - yz, z^2 - y w\}.$$ 
Let $f_1 : = y^2 - xz, f_2 := xw - yz$ and $f_3 : = z^2 - yw$. With our Python programme we obtain: 
\begin{align*} 
\operatorname{red}_R(\phi_R(f_1)) & = (2 a_{12} d_{12} - a_{13} - 2 b_{12} c_{12} + 4 b_{12} d_{12}^2 - 2 b_{12} d_{13} - 4 b_{13} d_{12} + 2 b_{14} - 4 c_{12}^2 d_{12} + 3 c_{12} c_{13} \\
&  \ \ \ - 2 c_{13} d_{12}^2 + c_{13} d_{13} + 2 c_{14} d_{12} - c_{15}) t^{23}; \\
\operatorname{red}_R(\phi_R(f_2)) & = (a_{12} + 2 b_{12} d_{12} - b_{13} + c_{12}^2 + 2 c_{12} d_{12}^2 - c_{14} + 2 d_{12}^2 d_{13} - 2 d_{12} d_{14} - d_{13}^2 + d_{15} ) t^{23}; \\
\operatorname{red}_R(\phi_R(f_3)) & = (- b_{12} - 3 c_{12} d_{12} + 2 c_{13} - 2 d_{12}^3 + 3 d_{12} d_{13} - d_{14} ) t^{23}. 
\end{align*} 

Let $g_1, g_2$ and $g_3$ denote the coefficients in front of $t^{23}$ in  $\operatorname{red}_R(\phi_R(f_1)),\operatorname{red}_R(\phi_R(f_2))$ and $\operatorname{red}_R(\phi_R(f_3))$ respectively. 
By \thref{step3}, $\mathcal{R}_{\Gamma} \leftrightarrow V(g_1,g_2,g_3) \subseteq \mathbb{C}^{15}$ since $15 = \sum_{i=0}^2 \text{\# gaps greater than $v_i$}.$ The coordinates of $\mathbb{C}^{15}$ are given by $(a_{12},\hdots, a_{23}, b_{12},\hdots, b_{23}, c_{12},\hdots, c_{23})$. Since $g_1$, $g_2$ and $g_3$ depend linearly on $b_{14}$, $d_{15}$ and $c_{13}$ respectively, and moreover since the linear components of the coefficients $g_i$ are linearly independent, it follows that the projection map $p: \mathbb{C}^{15} \rightarrow \mathbb{C}^{12}$ defined by omitting the variables $b_{14}, d_{15}$ and $c_{13}$ is invertible. 
Thus $$\mathcal{R}_{\Gamma} \leftrightarrow V(g_1,g_2,g_3) \leftrightarrow \mathbb{C}^{12}.$$

\end{example} 

\section{Properties of $\mathcal{R}_{\Gamma}$} \label{properties}

\subsection{When is $\mathcal{R}_{\Gamma}$ an affine space?} \label{alwaysaffine}
In each of the above examples, the affine variety corresponding to $\mathcal{R}_{\Gamma}$ is affine space $\mathbb{C}^{N(\Gamma)}$ for some $N(\Gamma)$ determined by numerical properties of the semigroup $\Gamma$. In the case of semigroups with two generators we know this to always be the case by \thref{2gen} since there are no elements in $S_{dec}(\Gamma)$ of smaller weighted order than the conductor. If there is just one element in $S_{dec}(\Gamma)$ of weighted order strictly less than the conductor, then $R_{\Gamma}$ is also an affine space, the dimension of which can be explicitly determined in terms of $\Gamma$ as per \thref{onegen} below. We note that the result that $\mathcal{R}_{\Gamma}$ is an affine space for such semigroups $\Gamma$ appears in \cite[Corollary 4]{Ishii1980}.

\begin{prop} \thlabel{onegen}
Let $\Gamma= \langle v_0, \hdots, v_g \rangle$ be a semigroup such that $S_{dec}$ contains a single polynomial strictly less than the conductor $c$ of $\Gamma$ and let $d$ denote its weighted order. Then $$\mathcal{R}_{\Gamma} \leftrightarrow \mathbb{C}^{N(\Gamma)} \subseteq \mathbb{C}^{M(\Gamma)}$$ where $$M(\Gamma) = \sum_{i=0}^g \text{ \# gaps of $\Gamma$ greater than $v_i$}$$ and
$$N(\Gamma) = M(\Gamma) - \text{\# gaps of $\Gamma$ greater than $d$}.$$
\end{prop} 

\begin{proof} 
The unique polynomial in $S_{dec}(\Gamma)$ of weighted order less than the conductor can be denoted by $$f = x_{i_1}^{k_1} \cdots x_{i_s}^{k_s} - x_{j_1}^{l_1} \cdots x_{j_t}^{l_t} \in \mathbb{C}[x_0,\hdots, x_g]$$ where $i_1,\hdots, i_s$ and $j_1, \hdots, j_t$ lie in $\{0,\hdots, g\}$ with $\{i_1, \hdots, i_s\} \cap \{j_1,\hdots, j_t\} = \emptyset$, all powers $k_i$ and $l_i$ are non-zero and $\sum_{n=1}^s v_{i_n} k_n = \sum_{n=1}^t v_{j_n} l_n < c$.

Let $R = \mathbb{C}[x_0(t), \hdots, x_g(t)]$ with $\{x_0(t),\hdots, x_g(t)\}$ in normal form with respect to $\Gamma$. Then $R \in \mathcal{R}_{\Gamma}$ if and only if $\operatorname{red}_R(\phi_R(f)) =  0$. Thus it suffices to analyse the coefficients of $\operatorname{red}_R(\phi_R(f)) =  0$, which we denote by $f^{(1)}, \hdots, f^{(p)}$. By tracking through polynomial multiplication and the reduction algorithm, it can be shown that each coefficient of $\operatorname{red}_R(\phi_R(f))$ depends linearly on the highest indexed coefficient of $x_{i_1}(t)$ appearing in its expansion as a polynomial in the coefficients of the generators $x_0(t), \hdots, x_g(t)$.

Thus in $\mathcal{R}_{\Gamma} = V(f^{(1)},\hdots, f^{(p)})$ the coefficients of $x_{i_1}(t)$ corresponding to gap powers of $\Gamma$ larger than $d$ can be expressed in terms of coefficients of the remaining $x_i(t)$ and of coefficients of $x_{i_1}(t)$ involving strictly smaller indices. It follows that $\mathcal{R}_{\Gamma}$ is in bijection with $\mathbb{C}^{N(\Gamma)}$ where $N(\Gamma) = M(\Gamma) - \text{\# gaps of $\Gamma$ greater than $d$}.$
\end{proof}

In general, when $S_{dec}(\Gamma)$ contains a larger number of polynomials of weighted order smaller than the conductor, explicitly identifying the affine variety corresponding to $\mathcal{R}_{\Gamma}$ becomes more difficult. To simplify the task, a logical first step is to search for a minimal generating set for $I_{dec}(\Gamma)$. In the case of semigroups with three generators, it can be shown that $I_{dec}(\Gamma)$ is minimally generated by exactly three polynomials.

\begin{prop} \thlabel{3gen}
Let $\Gamma = \langle v_0, v_1,v_2\rangle$. Let $k_0$ be the smallest integer such that the equation $k_0 v_0 = i v_1 + j v_2 \in \mathbb{Z}[i,j]$ admits a solution $(m_0,m_1) \in \mathbb{N}^2$, let $k_1$ be the smallest integer such that the equation $k_1 v_1  = i v_0 + j v_2 \in \mathbb{Z}[i,j]$ admits a solution $(n_0,n_1) \in \mathbb{N}^2$ and let $k_2$ be the smallest integer such that the equation $k_2 v_2 = i v_0 + n v_1 \in \mathbb{Z}[i,j]$ admits a solution $(p_0,p_1) \in \mathbb{N}^2$. Then $$I_{dec}(\Gamma) = (f_1, f_2, f_3)$$ where \begin{align*} 
f_1 & := x^{k_0} - y^{m_0} z^{m_1}; \\
f_2 & := y^{k_1} - x^{n_0} z^{n_1}; \\
f_3 & := z^{k_2} - x^{p_0} y^{p_1}.
\end{align*}
\end{prop}

\begin{proof} 
Since any element of $I_{dec}(\Gamma)$ can be expressed as a linear combination of deceptive homogeneous polynomial, to show that $I_{dec}(\Gamma) = (f_1, f_2,f_3)$ it suffices to show that $g \in \mathbb{C}[x,y,z]$ for any deceptive homogeneous polynomial $g = x^i y^j z^k - x^{i'} y^{j'} z^{k'}$, where $d = i v_0 + j v_1 + k v_2 = i' v_0 + j' v_1 + k' v_2$. Given such a deceptive polynomial $g$, one can factor out monomials from $g$ in order to obtain a product of a monomial and of a homogeneous deceptive polynomial of weighted degree strictly lower than $d$, unless $g$ was already one of the $f_i$. By induction we can then conclude that $g$ lies in $I_{dec}(\Gamma)$. 
\end{proof}

In each of the cases for semigroups with three generators which we have considered, the equations arising from the equalities $\operatorname{red}_R(\phi_R(f_i)) = 0$ for $i \in \{1,2,3\}$ have proven to be compatible in the sense that a number of `dependent' coefficients could be expressed in terms of the remaining `free' coefficients. This provided an identification of $\mathcal{R}_{\Gamma}$ with an affine space by omitting the dependent coefficients. Thus we pose the following question:

\begin{q} \thlabel{q1}
For semigroups with three generators, is $\mathcal{R}_{\Gamma}$ always an affine space? Morever, for general semigroups, is there a numerical criterion, extending \thref{onegen}, for determining when $\mathcal{R}_{\Gamma}$ is an affine space?  
\end{q} 

An example where the space $\mathcal{R}_{\Gamma}$ fails to be an affine space is provided in \cite[Example 3]{Ishii1980}, where it is stated that for $\Gamma = \langle 9,12,15,25,28,31 \rangle$ the space $\mathcal{R}_{\Gamma}$ is isomorphic to an irreducible hypersurface of degree $2$ in an affine space, which is singular.

\subsection{Stratification of $\mathcal{R}_{\Gamma}$ by embedding dimension} \label{stratificatn}

Given a semigroup $\Gamma = \langle v_0, \hdots, v_g \rangle$, any ring $R \in \mathcal{R}_{\Gamma}$ can be generated by $g+1$ polynomials $x_0(t), \hdots, x_g(t) $ in normal form with respect to $\Gamma$. Nevertheless, as noted in \thref{importantrk}, these polynomials may not minimally generate $R$. 

Recall from Section \ref{global} that we can interpret $R$ geometrically as the data of the complete local ring of a curve singularity, together with a choice of how to glue it on to a given smooth curve. From this perspective, if a ring $R$ is generated by $g+1$ polynomials, then a neighbourhood of the corresponding singular point can be embedded in $\mathbb{C}^{g+1}$. If $R$ can be generated by fewer, say $r < g+1$ polynomials, then the singularity can in fact be embedded in $\mathbb{C}^r \subsetneq \mathbb{C}^{g+1}$. In particular, if $R$ can be generated by just two polynomials, then we can think of $R$ as corresponding to a plane curve singularity. This motivates the following definition. 

\begin{defn}
Let $\Gamma = \langle v_0, \hdots, v_g \rangle$. We denote by $\mathcal{R}_{\Gamma}^{\text{plane}}$ the subset of $\mathcal{R}_{\Gamma}$ consisting of rings $R$ which can be generated by just two polynomials. More generally, for any $3 \leq n \leq g+1$, we denote by $\mathcal{R}_{\Gamma}^{(n)}$ the subset of rings $R$ which can be generated by $n$ elements. 
\end{defn}

\begin{remark}
Note that these subsets could well be empty for certain semigroups $\Gamma$. Indeed, not all semigroups $\Gamma$ can arise as the semigroup of a plane curve singularity for example. To see this, take $\Gamma = \langle 4, 6, 11 \rangle$ which has conductor $14$. If a ring $R = \mathbb{C}[x(t),y(t),z(t)] \in \mathcal{R}_{\Gamma}$ to be in $\mathcal{R}_{\Gamma}^{\text{plane}}$, we must have that $z(t) \in \mathbb{C}[x(t),y(t)]$ by order considerations. That is, $z(t)$ must be the image under $\phi_R$ of a polynomial $z(x,y) \in \mathbb{C}[x,y]$: $z(t) = z(x(t),y(t))$. Since $z(t)$ has order $11 \notin \langle 4,6 \rangle$, the polynomial $z$ must be deceptive. However, the smallest weighted degree of an element of $S_{dec}(\Gamma)$ is $12$, the least common multiple of $4$ and $6$. Thus $z(t)$ must have order strictly greater than $12$, a contradiction. Therefore the semigroup $\Gamma = \langle 4,6,11 \rangle$ is not the semigroup of a plane curve singularity, that is, $\mathcal{R}_{\Gamma}^{\text{plane}} = \emptyset$. 

In fact, there exist explicit criteria for determining whether or not a given semigroup $\Gamma$ can arise as the semigroup of a plane curve singularity. Such criteria were first obtained by Teissier in \cite[Appendix]{Zariski1986}, and later presented in a simpler form in \cite{Barucci2003}: a semigroup $\Gamma = \langle v_0 ,\hdots, v_g \rangle$ is the semigroup of a plane curve singularity, that is it satisfies $\mathcal{R}_{\Gamma}^{\text{plane}} \neq \emptyset$, if and only if the following two conditions hold:

\begin{enumerate}[(i)]
\item Let $e_i = \operatorname{gcd}(v_0, \hdots, v_{i})$ for all $i \in \{1,\hdots, g\}$. Then $e_1 > e_2 > \cdots > e_g = 1$; 
\item $v_i > \operatorname{lcm}(e_{i-2}, v_{i-1})$ for all $i \in \{2, \hdots, g\}$. 
\end{enumerate} 

In general however, given $d <  g+1$, no such criteria are known for determining whether or not $\mathcal{R}_{\Gamma}^{(d)}$ is empty, that is, whether or not $\Gamma$ can arise as the semigroup of a curve singularity of embedding dimension $d$ \cite[Problem 2.4]{Castellanos2005}.  
\end{remark}

Nevertheless, if $\Gamma$ is the semigroup of a plane curve singularity (i.e.\ satisfying (i) and (ii) above), we have a non-trivial stratification of the space $\mathcal{R}_{\Gamma}$: $$ \emptyset \neq \mathcal{R}_{\Gamma}^{\text{plane}} \subseteq \mathcal{R}_{\Gamma}^{(3)} \subseteq \cdots \subseteq \mathcal{R}_{\Gamma}^{(g)} \subseteq \mathcal{R}_{\Gamma}.$$ 

Given a ring $R= \mathbb{C}[x_0(t),\hdots, x_g(t)]$ with $\{x_0(t), \hdots, x_g(t)\}$, by \thref{step2} we have an algorithm for determining whether or not $R \in \mathcal{R}_{\Gamma}$. Similarly, one may ask if an algorithm can be found to determine when such a ring $R$ lies in $\mathcal{R}_{\Gamma}^{\text{plane}}$. The following example suggests an  approach one might take. 

\begin{example}
Let $\Gamma = \langle 4,6,13 \rangle$ and let $R \in \mathcal{R}_{\Gamma}$. As seen in \thref{4613Sdec}, $R$ can be written in the form $\mathbb{C}[x(t),y(t),z(t)]$ where \begin{align*} 
x(t) & = t^4 + a_5 t^5 + a_7 t^7 + a_9 t^9 + a_{11}t^{11} +a_{15} t^{15}, \\
y(t) & = t^6 + b_7 t^7 + b_9 t^9 + b_{11} t^{11} + b_{15} t^{15} \text{ and } \\
z(t) & = t^{13} + c_{15} t^{15}
\end{align*} satisfy the following: $$b_9 = - \frac{1}{2} (5 a_5^3 + 3 a_5^2 b_7 - 2 a_5 b_7^2 + 3 a_5 c_{15} - 3 a_7 - b_7^3 - 2 b_7 c_{15} + 2 b_9).$$

We claim that $R \in \mathcal{R}_{\Gamma}^{\text{plane}}$ if and only if $\phi_R(y^2 - x^3) = y(t)^2 - x(t)^3$ has order $13$. It is easy to see that if the latter is true, then $R$ is generated by just two elements. Indeed, if $y(t)^2- x(t)^3$ has order $13$, then removing powers larger than $13$ lying in $\Gamma$ using the method described in the proof of \thref{step1}, we must obtain the polynomial $z(t)$ by uniqueness of the triple $\{x(t),y(t),z(t)\}$ in normal form. Thus $z(t) \in \mathbb{C}[x(t),y(t)]$, and so $R \in \mathcal{R}_{\Gamma}^{\text{plane}}$. 

Conversely, suppose that $R \in \mathcal{R}_{\Gamma}^{\text{plane}}$. Then by order considerations, the only possible pair of generators which can generate $R$ is $(x(t),y(t))$. So $R = \mathbb{C}[x(t),y(t)]$. Suppose now, in order to reach a contradiction, that $\operatorname{ord}( y(t)^2 - x(t)^3) > 13$.  Since $z(t) \in \mathbb{C}[x(t),y(t)]$, there must exist some $q \in \mathbb{C}[x,y]$ such that $\phi_R(q) \in \mathbb{C}[x(t),y(t)]$ has order $13$. We can write $q = q_d + q_{d+1} + \hdots$ where each $q_i$ is a homogeneous polyomial of weighted degree $i$. Moreover, we can choose $q$ so that $d$ is maximal, that is, so that no other polynomial $q'$ of weighted order $d' > d$ can satisfy $\operatorname{ord} \phi_R(q') = 13$. Thus we can assume $\operatorname{ord} q_d = 13$. 

Since $\phi_R(q)$ has order $13 \notin \langle 4,6\rangle$, the polynomial $q_d$ must be deceptive and as it is homogeneous, it lies in $I_{dec}(\Gamma)$. As we have seen in \thref{4613Sdec}, the only generator of $I_{dec}(\Gamma)$ of weighted degree smaller than the conductor is $y^2 - x^3$. Thus $q_d = r(y^2 - x^3)$ for some $r \in \mathbb{C}[x,y]$. If $r$ is a constant, then $13 = \operatorname{ord} q_d = \operatorname{ord} (y^2 - x^3) > 13$, a contradiction. Otherwise, we have: $$\operatorname{ord} \phi_R(q) = \operatorname{ord} \phi_R(q_d) = \operatorname{ord} \phi_R(r(y^2 - x^3)) 
= \operatorname{ord} \phi_R(r) \times \operatorname{ord} \phi_R(y^2 - x^3) > 13,$$ which again contradicts the assumption that $\operatorname{ord} \phi_R(q) = 13$. Hence $y(t)^2 - x(t)^3$ has order $13$. This proves the claim.  

Direct computation yields: $$y(t)^2 - x(t)^3 = (-3 a_5 + 2 b_7) t^{13} + (-3 a_5^2 + b_7^2) t^{14} + (-a_5^3 - 3 a_7 + 2 b_9)t^{15}.$$ Thus we have: $R \in \mathcal{R}_{\Gamma}^{\text{plane}}$ if and only if $b_7 \neq \frac{3}{2} a_5$. 

As seen in \thref{4613}, we can identify $\mathcal{R}_{\Gamma}$ with $\mathbb{C}^9$ via the map sending $R$ to the $9$-tuple of coefficients $(a_5,a_7,a_9,a_{11},a_{15},b_7,b_{11},b_{15}, c_{15})$. Since from above $R \in \mathcal{R}_{\Gamma}$ if and only if $b_7 \neq \frac{3}{2} a_5$, it follows that we can identify $\mathcal{R}_{\Gamma}^{\text{plane}}$ with the subset $\mathbb{C}^{\ast} \times \mathbb{C}^8$ of $\mathbb{C}^9$ via the map sending $R$ to $(2 b_7 - 3 a_5, a_5,_7,a_9,a_{11},a_{15},b_{11},b_{15}, c_{15})$. Thus $$\mathcal{R}_{\Gamma}^{\text{plane}} \leftrightarrow \mathbb{C}^{\ast} \times \mathbb{C}^8 \subseteq \mathbb{C}^9 \leftrightarrow \mathcal{R}_{\Gamma}.$$

\end{example}
 
The above example can be generalised to arbitrary semigroups with three generators. That is, given $\Gamma = \langle v_0, v_1, v_2 \rangle$ and $\{x(t),y(t),z(t)\}$ in normal form, $R = \mathbb{C}[x(t),y(t),z(t)]  \in \mathcal{R}_{\Gamma}^{\text{plane}}$ if and only if $\operatorname{red}_{\langle v_0, v_1 \rangle} (\phi_R(y^{k_1} - x^{k_0}))$ has order $v_2$. Analysing the form of the coefficient in front of $t^{v_2}$ in $\operatorname{red}_{\langle v_0, v_1 \rangle} \phi_R(y^{k_1} - x^{k_0})$ then shows that if $\mathcal{R}_{\Gamma} \leftrightarrow \mathbb{C}^N$, the  space $\mathcal{R}_{\Gamma}^{\text{plane}}$ can be identified with $\mathbb{C}^{\ast} \times \mathbb{C}^N$. 

For arbitrary semigroups $\Gamma = \langle v_0, \hdots, v_g \rangle$ the following two questions remain. 

\begin{q} 
Can we explicitly write down polynomials $f_d,\hdots, f_g \in \mathbb{C}[x_0, \hdots, x_g]$ such that $R = \mathbb{C}[x_0(t), \hdots, x_g(t)]$ lies in $\mathcal{R}_{\Gamma}^{(d)}$ if and only if $\operatorname{ord} \operatorname{red}_{\langle v_0, \hdots, v_{i-1} \rangle} (\phi_R(f_i)) = v_i$ for all $i \in \{d, \hdots, g\}$? 
\end{q} 

If so, then by letting $h_i$ denote the leading coefficient of $\operatorname{red}_{\langle v_0, \hdots, v_{i-1} \rangle} (\phi_R(f_i))$, we would obtain: $$\mathcal{R}_{\Gamma}^{(d)} \leftrightarrow D(h_2,\hdots, h_g) \subseteq \mathcal{R}_{\Gamma}.$$ That is, $\mathcal{R}_{\Gamma}^{(d)}$ is an open subvariety of $\mathcal{R}_{\Gamma}$.

\begin{q}
If $\mathcal{R}_{\Gamma}$ is in bijection with an affine space $\mathbb{C}^N$, can we always identify $\mathcal{R}_{\Gamma}^{(d)}$ with the space ${(\mathbb{C}^{\ast})}^{g+1-d} \times \mathbb{C}^{N-g+d-1}$?
\end{q}

\subsection{The map from $\mathcal{M}_{\Gamma}$ to $\mathcal{R}_{\Gamma}$} \label{mapmtor}

In Section \ref{zariski}, we defined the Zariski moduli space to be the space of unibranch curve singularities with semigroup $\Gamma$ up to analytic equivalence. We then showed that $\mathcal{M}_{\Gamma}$ can be interpreted as the quotient of $\mathcal{R}_{\Gamma}$ by the action of $\operatorname{Aut} \mathbb{C}[[t]]$. Note that automorphisms of $\mathbb{C}[[t]]$ consist of power series of order one. The action of an element $\rho(t) \in \operatorname{Aut} \mathbb{C}[[t]]$ on a $\mathbb{C}$-subalgebra $R = \mathbb{C}[[x_0(t),\hdots, x_g(t)]] \subseteq \mathbb{C}[[t]]$ is defined by $\rho(t) \cdot R = \mathbb{C}[[ x_0(\rho(t)), \hdots, x_g(\rho(t))]]$.

Thus there is a quotient map from $\mathcal{R}_{\Gamma}$ to $\mathcal{M}_{\Gamma}$. In the two examples below, we explicitly compute this map.

\begin{example}

Let $\Gamma = \langle 3,7 \rangle$. As computed by Zariski in his monograph \cite{Zariski1986}, the space $\mathcal{M}_{\Gamma}$ consists of two points $\mathbb{C}[[t^3, t^7]]$ and $\mathbb{C}[[t^3, t^7+t^8]]$, with the former lying in the closure of the latter. By \thref{2gen}, we know that $\mathcal{R}_{\Gamma}$ is in bijection with $\mathbb{C}^6$ with coordinates $(a_4,a_5,a_8,a_{11}, b_8, b_{11})$. 

To compute the map $\mathcal{R}_{\Gamma} \rightarrow \mathcal{M}_{\Gamma}$ we must determine which conditions on the coefficients $a_i$ and $b_i$ ensure that the corresponding ring $R \in \mathcal{R}_{\Gamma}$ is isomorphic to $\mathbb{C}[[t^3, t^7]]$ rather than $\mathbb{C}[[t^3, t^7+ t^8]]$. To do so, we simply follow Zariski's method for characterising $\mathcal{M}_{\Gamma}$, a method in three steps \cite[Proposition V.1.2]{Zariski1986}. First, we find an automorphism $\phi$ of $\mathbb{C}[[t]]$ which sends $x(t)$ to $t^3$. By recursively solving for the coefficients of an automorphism with this property, we can set: $$\phi(t) = t  - \frac{1}{3} a_4 t^2 + \frac{1}{3} (a_4^2 - a_5) t^3 + \cdots.$$ Thus we have: $$R \cong \phi(R) = \mathbb{C}[[t^3, \phi(y(t))]],$$ where: $$\phi(y(t)) = t^7 + \left( - \frac{7}{3} a_4 + b_8 \right) t^8 + \cdots$$ by direct computation.

The second step is to remove as many powers of $\phi(y(t))$ as possible whilst ensuring that the generated ring stays in the same isomorphism class, i.e.\ that the resulting parametrisation is in the same $\mathcal{A}$-equivalence class (see Section \ref{zariski}). This is achieved via Zariski's ``elimination criteria'', neatly summarised in \cite{HH2007}, which determine when certain powers appearing in $\phi(y(t))$ can be removed whilst preserving $\mathcal{A}$-equivalence. In the present example, the elimination criteria imply that all of the terms of orders $9$ and higher can be removed from $\phi(y(t))$. Thus: $$\phi(R) \cong \mathbb{C}[[t^3, t^7 + \left( - \frac{7}{3} a_4 + b_8 \right) t^8 ]].$$ 

The third step is to observe that if $- \frac{7}{3} a_4 + b_8 \neq 0$, then: $$\mathbb{C}[[t^3, t^7 + \left( - \frac{7}{3} a_4 + b_8 \right) t^8 ]] \cong \mathbb{C}[[t^3, t^7  + t^8]]$$ under an automorphism of $\mathbb{C}[[t]]$ of the form $t \mapsto \alpha t$ for an appropriately chosen $\alpha \in \mathbb{C}$. It follows that $R$ is isomorphic to $\mathbb{C}[[t^3, t^7]]$ if and only if $b_8 = \frac{7}{3} a_4$. The map $\mathcal{R}_{\Gamma} \to \mathcal{M}_{\Gamma}$ can therefore be written down explicitly: 

$$\mathbb{C}[[t^3+a_4 t^4+ a_5 t^5+ a_8 t^8 + a_{11} t^{11}, t^7 + b_8 t^8 + b_{11} t^{11}]] \mapsto \begin{cases}
\mathbb{C}[[t^3, t^7]] \text{ if $b_8 \neq \frac{7}{3} a_4$} \\
\mathbb{C}[[t^3, t^7+ t^8]] \text{ if $b_8 = \frac{7}{3} a_4$}.
\end{cases}$$ 

\end{example}

\begin{example}Let $\Gamma = \langle 4,9 \rangle$. Applying Hefez and Hernandez's elimination criteria from \cite{HH2007}, we know that $\mathcal{M}_{\Gamma}$ consists of three components: 
\begin{align*}
&\text{I }: \mathbb{C}[[t^4, t^9 + t^{10} + c t^{11} ]]; \\
& \text{II }: R_0 = \mathbb{C}[[t^4, t^9 ]]; \\
& \text{III }: R_1 = \mathbb{C}[[t^4, t^9 + t^{11} ]].
\end{align*}
The $\mathbb{C}$-subalgebras of type I are in bijection with points of $\mathbb{C}$. Thus $\mathcal{M}_{\Gamma} = \mathbb{C} \cup \{R_0, R_1\}$. 

By \thref{2gen}, we know that $\mathcal{R}_{\Gamma}$ is in bijection with $\mathbb{C}^{17}$ with coordinates: $$\mathbf{a} = (a_5, a_6, a_7, a_{10}, a_{11}, a_{14}, a_{15}, a_{19}, a_{21},a_{23}, b_{10}, b_{11}, b_{14}, b_{15}, b_{19}, b_{21}, b_{23}) \in \mathbb{C}^{17}.$$ As in the previous example, we can apply Zariski's method in three steps, to obtain that the map $\mathcal{R}_{\Gamma} \to \mathcal{M}_{\Gamma}$ is given by: $$ \mathbf{a} \mapsto \begin{cases}
\frac{135 a_5^2 - 72 a_6 - 80 a_5 b_{10} + 32 b_{11}}{2 ( 9 a_5 - 4 b_{10})^2} & \text{ if $9 a_5 \neq 4 b_{10}$;} \\
R_0 & \text{ if $ 9 a_5 = 4 b_{10} $ and $b_{11} = - \frac{9}{32} a_5^2 - \frac{9}{4} a_6 - \frac{5}{2} a_5 b_{10}$;} \\
R_1 &  \text{ otherwise}.
\end{cases} $$

Thus on $D(9 a_5 - 4 b_{10}) \subseteq \mathcal{R}_{\Gamma}$, the function $\frac{135 a_5^2 - 72 a_6 - 80 a_5 b_{10} + 32 b_{11}}{2 ( 9 a_5 - 4 b_{10})^2}$ is invariant under the action of $\operatorname{Aut} \mathbb{C}[[t]]$. 
\end{example} 

It would be interesting if the theory of non-reductive Geometric Invariant Theory could be applied to construct the GIT quotient $\mathcal{R}_{\Gamma} /\!\!/ \operatorname{Aut} \mathbb{C}[[t]]$, which would be a separated, geometrically tractable subset of the Zariski moduli space $\mathcal{M}_{\Gamma}$. For example, in the case considered above, it appears that $\mathcal{R}_{\Gamma}^{\text{ss}}$ should correspond to the locus of points determined by the equation $9a_5 \neq 4 b_{10}$.

\bibliographystyle{abbrv}

\begin{thebibliography}{10}

\bibitem{Smyth2016}
J.~Alper, D.~Smyth, M.~Fedorchuk, and F.~V. der Wyck.
\newblock {S}econd {F}lip in the {H}assett-{K}eel {P}rogram: {E}xistence of
  good moduli spaces.
\newblock {\em Compositio Mathematica}, 2016.

\bibitem{Barucci2003}
V.~Barucci, M.~D'Anna, and R.~Fr\"oberg.
\newblock On {P}lane {A}lgebroid {C}urves.
\newblock {\em Lecture Notes in Pure and Applied Mathematics}, 231:37--50,
  2003.
\newblock arXiv:math/0302224 [math.AC].

\bibitem{Castellanos2005}
J.~Castellanos.
\newblock The {S}emigroup of a {S}pace {C}urve {S}ingularity.
\newblock {\em Pacific {J}ournal of {M}athematics}, 221(2), 2005.

\bibitem{Delorme1978}
C.~Delorme.
\newblock Sur les modules des singularit\'es des branches planes.
\newblock {\em Bulletin de la Soci\'et\'e Math\'ematique de France},
  106:417--446, 1978.

\bibitem{Grothendieck1960-1961}
A.~Grothendieck.
\newblock Techniques de construction en g\'eom\'etrie analytique. {VI}. etude
  locale des morphismes: germes d'espaces analytiques, platitude, morphismes
  simples.
\newblock {\em S\'eminaire Henri Cartan}, 13:1--13, 1960-1961.

\bibitem{HH2007}
A.~Hefez and M.~Hernandez.
\newblock The analytic classification of plane branches.
\newblock arXiv:0707.4502 [math.AG], 2007.

\bibitem{Ishii1980}
S.~Ishii.
\newblock Moduli of {S}ubrings of a {L}ocal {R}ing.
\newblock {\em Journal of Algebra}, 67:504--516, 1980.

\bibitem{Laudal1988}
O.~A. Laudal, B.~Martin, and G.~Pfister.
\newblock Moduli of plane curve singularities with $\mathbb{C}^{\ast}$-action.
\newblock {\em Singularities, {B}anach {C}enter {P}ublications}, 20, 1988.

\bibitem{Luengo1990}
I.~Luengo and G.~Pfister.
\newblock Normal forms and moduli spaces of curve singularities with semigroup
  $\langle 2p,2q,2pq+d \rangle$.
\newblock {\em Compositio Mathematica}, 76:247--264, 1990.

\bibitem{Smyth2011}
D.~Smyth.
\newblock {M}odular compactifications of the space of pointed elliptic curves
  {I}.
\newblock {\em Compositio Mathematica}, 147(3):877--913, 2011.

\bibitem{Smyth2011a}
D.~Smyth.
\newblock Modular compactifications of the space of pointed elliptic curves
  {II}.
\newblock {\em Compositio Mathematica}, 147(6):1843--1884, 2011.

\bibitem{Waldi1972}
R.~Waldi.
\newblock {\em Wertehalbgruppe und Singularit\:at einer eben en algebraischen
  Kurve, Dissertation}.
\newblock PhD thesis, Regensburg Univ., 1972.

\bibitem{Washburn1988}
S.~Washburn.
\newblock Book review: Le probl\`eme des modules pour les branches planes by
  {O}scar {Z}ariski with an appendix by {B}ernard {T}eissier.
\newblock {\em Bulletin (New Series) of the American Mathematical Society},
  18(2), April 1988.

\bibitem{Zariski1965}
O.~Zariski.
\newblock Studies in {E}quisingularity {I}, {E}quivalent singularities of plane
  algebroid curves.
\newblock {\em American Journal of Mathematics}, 87:507--536, 1965.

\bibitem{Zariski1986}
O.~Zariski.
\newblock {\em Le probl\`eme des modules pour les branches planes, with an
  appendix by Bernard Teissier}.
\newblock Hermann, Paris, 1986.

\end{thebibliography}

\end{document}